\documentclass[12pt]{article}

\usepackage[toc,page]{appendix}

\usepackage[utf8]{inputenc}
\usepackage{amssymb, amsfonts, amsmath, amsthm}
\usepackage{enumerate}
\usepackage{xcolor}

\usepackage{pgf,tikz}

\usepackage{fullpage}
\usepackage[colorlinks=true,linkcolor=blue,citecolor=blue]{hyperref}

\usepackage{libertine}
\usepackage{libertinust1math}
\usepackage[T1]{fontenc}

\usepackage{verbatim}

\newcommand{\G}{\mathcal{G}}
\newcommand{\snc}[2]{S_{NC}[#2;#1]}
\newcommand{\intpart}{\zeta}

\newcommand{\cyc}[1]{\# (#1)}
\newcommand{\ip}[1]{ |#1|}
\newcommand{\blocks}[1]{\# #1}

\newcommand{\partlat}{P}
\newcommand{\nc}{NC}
\newcommand{\freec}{\kappa}

\newcommand{\herm}{\hat{H}}
\newcommand{\lag}{\hat{L}}

\newcommand{\p}{r}
\newcommand{\q}{s}

\newcommand{\calP}{\mathcal{P}}

\newcommand{\R}{\mathbb{R}}

\newcommand{\N}{\mathbb{N}}

\newcommand{\rrad}{\mathcal{M}_{\mathcal{R}}}
\newcommand{\monicpols}{\mathbb{P}_{mon}(d)}
\newcommand{\monicpolreal}{\mathcal{R}_{mon}(d)}
\newcommand{\monicpolplus}{\mathcal{R}^+_{mon}(d)}

\newtheorem{theorem}{Theorem}[section]
\newtheorem{lemma}[theorem]{Lemma}
\newtheorem{definition}[theorem]{Definition}

\newtheorem{proposition}[theorem]{Proposition}

\newtheorem{notation}[theorem]{Notation}
\newtheorem{remark}[theorem]{Remark}

\newtheorem{example}[theorem]{Example}

\title{Finite Free Cumulants: Multiplicative Convolutions, Genus Expansion and Infinitesimal Distributions}

\author{\normalsize Octavio Arizmendi \thanks{O.A. gratefully acknowledges financial support by the grants Conacyt A1-S-9764 and SFB TRR 195.} \\ \normalsize octavius@cimat.mx  \\ \normalsize CIMAT, Guanajuato \and \normalsize Jorge Garza-Vargas\thanks{Supported by the NSF grant CCF-2009011.}\\\normalsize jgarzavargas@berkeley.edu\\ \normalsize UC Berkeley \and \normalsize Daniel Perales \thanks{Supported by CONACyT (Mexico) via the scholarship 714236.} \\\normalsize dperales@uwaterloo.ca\\ \normalsize University of Waterloo} 

\date{}

\begin{document}
\maketitle

\begin{abstract}Given two polynomials $p(x), q(x)$ of degree $d$, we give a combinatorial formula for the finite free cumulants of $p(x)\boxtimes_d q(x)$. We show that this formula admits a topological expansion in terms of non-crossing multi-annular permutations on surfaces of different genera. 

This topological expansion, on the one hand,  deepens the connection between the theories of finite free probability and free probability, and in particular   proves that $\boxtimes_d$ converges to $\boxtimes$ as $d$ goes to infinity. On the other hand, borrowing tools from the theory of second order freeness, we use our expansion to study the infinitesimal distribution of certain families of polynomials which include Hermite and Laguerre, and draw some connections with the theory of infinitesimal distributions for real random matrices.   

Finally, building on our results we give a new short and conceptual proof of a recent result  \cite{steinerberger2020, hoskins2020dynamics} that connects root distributions of polynomial derivatives with free  fractional convolution powers. 
\end{abstract}

\iffalse
\begin{itemize}
    \item 
Finite Free Cumulants: Multiplicative Convolutions and Asymptotic Expansion

\item  Finite Free Cumulants: Multiplicative Convolutions, genus expansion and infinitesimal distributions.

\item Genus Expansion of Finite Free Cumulants: Multiplicative Convolutions, Fractional Powers and Infinitesimal Distributions.

\item Finite Free Convolutions Revisited: multiplication, genus expansion and free fractional powers

\item Finite Free Multiplicative Convolution and Free Fractional Powers.

\item Finite free cumulants revisited: multiplicative convolutions, fractional powers and infinitesimal distributions. 

\item Genus expansion of finite free cumulants

\item Genus expansion of finite free cumulants: multiplicative convolution, fractional powers and infinitesimal distributions. 
\end{itemize}
\fi

\section{Introduction}

 The connection between free probability and random matrices, discovered by Voiculescu \cite{voiculescu1991limit}, is nowadays well-known and has found a broad range of applications. One could roughly summarize this connection as follows: large independent randomly rotated matrices behave like free random variables. This means that when the dimension $d$ goes to infinity, we may calculate the distribution of polynomials in these matrices by plugging in free random variables.

The topic we are concerned with here is the connection between free probability and the analytic theory of polynomials, where already for polynomials of fixed degree $d$ interesting parallels between the two theories emerge. This was recently brought to light by Marcus, Spielman and Srivastava \cite{marcus2016polynomial}, where two classical polynomial convolutions \cite{szego1922bemerkungen, walsh1922location} were rediscovered as expected characteristic polynomials of the sum and product of randomly rotated matrices. Since then, the \emph{finite free convolutions} discussed in the aforementioned work, and their relevance in  free probability and the analytic theory of polynomials, were further explored by Marcus \cite{marcus}, and  have been revisited in the now growing literature of \emph{finite free probability} \cite{arizmendi2018cumulants, leake2018further, gribinski2019rectangular,rosenmann2019polynomial,campbell2020finite,mirabelli2021hermitian}.

%In terms of random matrices, we are not directly interested in the eigenvalue distribution of a random matrix but rather in their expected characteristic polynomial. 
%  and their approximation to a limit distribution as the dimension $d$ tends to $\infty$.

%Following \cite{marcus2016polynomial}, where two classical polynomial convolutions were rediscovered as expected characteristic polynomials of the sum and product of randomly rotated matrices. 

%The connection with free probability was further developed by Marcus \cite{marcus}. Here we are mostly interested on the multiplicative case.
%Let $p_n$ and $q_n$, be two real-rooted polynomials the finite free additive convolution of $p_n$ and $q_n$ is given by

%If $p_n$ has positive root then the finite free additive convolution of $p_n$ and $q_n$. 

%Let $A_n$ and $B_n$Explain...

%Given their resemblance with free convolutions, these operations with where 

%a framework in which the free probability was observed in Marcus....

%In particular, one may deduce that finite free additive convolution approaches free additive convolution. The consideration in \cite{marcus} also suggests that finite free multiplicative convolution approaches free multiplicative convolution. We 

Of greatest relevance to this paper is the combinatorial approach, based on cumulants, for the finite free additive convolution that was introduced by Arizmendi and Perales \cite{arizmendi2018cumulants}. These \emph{finite free cumulants} converge to free cumulants as $d$ goes to infinity and  share many of their properties. In the present work we further this approach to also include the finite free multiplicative convolution in the combinatorial description and present  applications of our results to the asymptotic theory of polynomials. Below we give a brief summary of our main results, deferring to Section \ref{sec:allpreliminaries} the precise definitions of some of the notation used in these statements. 
\bigskip

\noindent \textbf{Combinatorial formulas.} Following \cite{marcus2016polynomial}, given two complex monic polynomials
\[
p(x)=\sum_{i=0}^d x^{d-i}(-1)^i a^p_i \qquad \text{and} \qquad  q(x)=\sum_{i=0}^d x^{d-i}(-1)^i  a^q_i
\]
of degree at most $d$, the \emph{finite  free additive convolution}, $p\boxplus_d q$, and the \emph{finite free multiplicative  convolution}, $p\boxtimes_d q$, are defined as:
\begin{align*}
[p\boxplus_d q](x)&:=\sum_{k=0}^d x^{d-k}(-1)^k \sum_{i+j=k}\frac{(d-i)!(d-j)!}{d!(d-k)!}a_i^p a_j^q , 
\end{align*}
and
$$[p\boxtimes_d q](x): =\sum_{k=0}^d x^{d-k}(-1)^k \frac{a_k^p a_k^q }{\binom{d}{k}}.$$
Our first result gives a formula for the finite free cumulants  $\kappa_n^d$ of $p\boxtimes_d q$ in terms of the finite free cumulants of $p$ and $q$ (see Definition \ref{defi.finite.free.cumulants}), and for the moments $m_n$ of the empirical root distribution of $p\boxtimes_d q$ (defined in \eqref{defi.moments}) in terms of the finite free cumulants of $p$ and the moments of $q$. We refer the reader to Section \ref{sec:allpreliminaries} for definitions of the notation related to partitions. 

\begin{theorem}[Primary formulas]
\label{thm.cumulant.of.products}
Let $p$ and $q$ be monic polynomials of degree $d$.  Then, the following formulas hold:
\begin{equation*}
%\label{eq.cumulant.of.products}
\kappa_n^d(p \boxtimes_d q )=\frac{(-1)^{n-1}}{d^{n+1} (n-1)!} \sum_{\substack{ \sigma, \tau\in P(n) \\ \sigma\lor  \tau=1_n }} d^{\blocks{\sigma}+\blocks{\tau}}\mu(0_n,\sigma)\mu(0_n,\tau) \kappa_\sigma^d(p ) \kappa^d_\tau(q )
\end{equation*}
and 
\begin{equation*}
%\label{eq:multiplicativemomentcumulant}
m_n(p \boxtimes_d q )=\frac{(-1)^{n-1}}{d^{n+1} (n-1)!} \sum_{\substack{ \sigma, \tau\in P(n) \\  \sigma\lor \tau=1_n }} d^{\blocks{\sigma}+\blocks{\tau}}\mu(0_n,\sigma)\mu(0_n,\tau) \kappa_\sigma^d(p ) m_\tau(q ).
\end{equation*}

\end{theorem}

Note that as $d$ goes to infinity, the leading terms in the right-hand side of the above equations correspond to those pairs of partitions satisfying $\blocks{\sigma}+\blocks{\tau} = n+1$.  Using a formula that rewrites sums over  these pairs of partitions as sums over non-crossing partitions $\pi$ and their Kreweras complement $Kr(\pi)$, we will prove the following. 
 
\begin{theorem}[Terms of order $\Theta(1)$] \label{thm2} For any $p$ and $q$ monic polynomials of degree $d$,
\[
\kappa_n^d(p\boxtimes_d q)= \sum_{\pi\in NC(n)} \kappa_\pi^d(p) \kappa_{Kr(\pi)}^d(q)+ O(1/d)
\]
and 
\[m_n(p\boxtimes_d q)= \sum_{\pi\in NC(n)} \kappa_\pi^d(p) m_{Kr(\pi)}(q)+ O(1/d).\]
\end{theorem}

This will prove useful in analyzing the asymptotic behavior of $\boxtimes_d$. In particular, the above theorem implies that our cumulant-cumulant formula coincides, up to the first order, with the formula obtained by Nica and Speicher for the cumulants of a product of two free random variables \cite[Theorem 1.4]{nica1996multiplication} (also see \cite{krawczyk2000combinatorics}). 

Moreover, when $q(x)=(x-1)^d$, the second equation in Theorem \ref{thm2} recovers the first-order asymptotics of $\kappa_n^d$ proven in \cite{arizmendi2018cumulants}, namely
$$m_n(p_d) = \sum_{\pi\in NC(n)} \kappa_\pi^d(p)+O(1/d).$$
This formula is helpful when studying the limiting bulk behavior of root distributions.  However, when working with ``fluctuations" of empirical root distributions around their limiting measure, it will also be necessary to have a better understanding of the terms of order $\Theta(1/d)$ in the moment-cumulant formula. In Section \ref{sec:generalformula} we will show that these terms can be written as a sum over the set of  annular non-crossing permutations $S_{NC}(\p, \q)$ introduced in \cite{mingo2004annular} (see Section \ref{sec:annularpermutations} for a definition). In particular, we obtain the following result.

\begin{theorem}[Terms of order $\Theta(1/d)$]
\label{thm:3}
Let $p$ be a monic polynomial of degree $d$ and let $n\leq d$, then 
\begin{equation}\label{eqthm13}m_n(p) - \sum_{\pi\in NC(n)} \kappa_\pi^d(p) = - \, \frac{n}{2 d}  \sum_{ \substack{\p+\q=n\\ \pi \in S_{NC}(\p,\q)}} \frac{\kappa_\pi^d(p)}{\p\q}\, +\,O(1/d^2).\end{equation}
\end{theorem}

Theorems \ref{thm2} and \ref{thm:3} are a consequence of our main combinatorial result, which we state and prove in Section  \ref{sec:genusformula}. In short, for the general case rather than summing over partitions we sum over  permutations and, using the notion of relative genus and some of the theory of maps and surfaces, for every $k\geq 0$, we give a topological expansion for the terms of order $\Theta(1/d^k)$ appearing in the above formulas (see Theorem \ref{thm:generalformula}  for a precise statement).  Interestingly, the terms of genus zero in our expansion for the order $\Theta(1/d^k)$ are precisely the (planar) multi-annular non-crossing permutations with $k+1$ circles. These combinatorial objects were introduced in \cite[Section 8]{mingo2004annular} and appear naturally in the theory of second (and higher) order freeness  \cite{mingo2006second, mingo2007second,  collins2006second}.

\bigskip

\noindent \textbf{Asymptotic root distributions.} Theorem \ref{thm2} given above allows us to obtain new proofs of two facts relating the asymptotic behavior of certain families of polynomials with operations in free probability.  

%Firstly, as the expert reader may predict, we can use our result to show that the finite free multiplicative convolution $\boxtimes_d$ approaches the free multiplicative convolution $\boxtimes$. This phenomenon was discovered by Marcus   in \cite{marcus}, where he showed that  a transform that linearizes $\boxtimes_d$ converges, as $d$ goes to infinity, to the logarithm of Voiculescu's $S$-transform, which is known to linearize $\boxtimes$. In the present paper we give a statement in terms of convergence of measures. 

Firstly, as the expert reader may predict, we can use our result to show that the finite free multiplicative convolution $\boxtimes_d$ converges to the free multiplicative convolution $\boxtimes$. The fact that the finite free multiplicative convolution is related in the limit to the free multiplicative convolution was  discovered by Marcus in \cite{marcus}, where he showed that  a transform that linearizes $\boxtimes_d$ (i.e. the logarithm of the $d$-finite $S$-transform) converges, as $d$ goes to infinity, to the logarithm of Voiculescu's $S$-transform, which is known to linearize $\boxtimes$. In the present paper we give a statement in terms of convergence of measures. Note that we do not need to restrict to the case when both measures are supported on $[0, \infty)$ and we do not require the sequences of polynomials to have a uniformly bounded root distribution.

\begin{theorem}[Weak convergence]
\label{thm:weakconvergence}
Let $\mu$ and $\nu$ be probability measures supported on a compact subset of the real line. Let $(p_d)_{d=1}^\infty$ and $(q_d)_{d=1}^\infty$ be sequences of monic real-rooted polynomials, where $p_d$ and $q_d$ have degree $d$, and assume that the $q_d$ have only non-negative roots. If the empirical root distributions of these sequences of polynomials converge weakly to $\mu$ and $\nu$ respectively, then the empirical root distributions of the sequence $(p_d\boxtimes_d q_d)_{d=1}^\infty$ converge weakly to $\mu\boxtimes \nu$. 
\end{theorem}

\begin{remark}
This theorem does not seem to follow from the results obtained in \cite{marcus}\footnote{The main obstacle being that the logarithm of the $d$-finite $S$-transform is only shown to linearize $\boxtimes_d$ at the points $\{-\frac{k}{d}\}_{k=0}^d$ (see Lemma 4.10 in \cite{marcus}), while the convergence to the logarithm of the $S$-transform is only proven for a subset of $(0, \infty)$ (see Lemma 4.11 in \cite{marcus}), and all of our attempts to use an analytic continuation argument faced technical difficulties that we did not see how to circumvent.}.
\end{remark}

Secondly, we derive the interesting relation between derivatives of a polynomial and free additive convolution powers, as observed by Steinerberger \cite{steinerberger2020} and proved by Hoskins and Kabluchko \cite{hoskins2020dynamics}. The main observation here, which we prove in Section \ref{sec:freefractionalpowers}, is that when $q(x)= x^{j} (x-1)^{d-j}$  it holds that
\begin{equation}
\label{eq:derivativeasmult}
p(x)\boxtimes_d q(x) = \frac{1}{(d)_j} x^{j} D^{j} p(x),   
\end{equation}
for any $p$ and where $D$ denotes differentiation with respect to $x$. So, in the setup where we have a sequence $(p_d)_{d=1}^\infty$ of polynomials whose root distributions are  converging to some compactly supported limiting measure $\mu$, equation (\ref{eq:derivativeasmult}) can be used in combination with Theorem \ref{thm:weakconvergence} to study, from a free probability perspective, the root distribution of the polynomials obtained by repeatedly differentiating the $p_d$. See Section \ref{sec:freefractionalpowers} for precise statements.

Finally, in the same framework of a sequence of real-rooted polynomials $(p_d)_{d=1}^\infty$ with  asymptotic (compactly supported) root distribution $\mu$,  we use Theorem \ref{thm:3} to study the ``fluctuations" of order $\Theta(1/d)$ of the root distributions of the $p_d$ around $\mu$. To be precise, if $m_n$ are the moments of $\mu$ we can define
$$m_n' := \lim_{d\to\infty} d(m_n(p_d)-m_n), $$
when the limit exists, and let $\mu'$ be the signed measure with moments $m_n'$. In the free probability literature \cite{belinschi2012free, shlyakhtenko2015free, mingo2019non} the measure $\mu'$ is referred to as the \emph{infinitesimal distribution} of the sequence $(p_d)_{d=1}^\infty$, and the pair $(\mu, \mu')$  is studied through the transforms 
\[G_\mu(z):=\sum_{n\geq 0} m_n z^{-n-1} \quad \text{and}  \quad G_{inf}(z):=\sum_{n\geq 1} m'_n z^{-n-1}.\]

\iffalse

In this paper, we obtain that for certain families of polynomials, the Cauchy transform of $\mu'$ can be directly related to the Cauchy transform of $\mu$. This relation turns out to be related to the Markov transform  \cite{kerov1998}. If there exists a unique signed measure $M(\mu)$ which satisfies the relation $$G_{M(\mu)}(z)=-\frac{G'_\mu(z)}{G_\mu(z)},$$
we say that $M(\mu)$ is the inverse transform of $\mu$. 

{\color{red}In this paper we obtain that for certain families of polynomials, $\mu'$ can  be described in terms of the inverse Markov transform $M(\mu)$ of $\mu$, \cite{kerov1998}, which is defined to be the unique signed measure that satisfies the functional equation $$G_{M(\mu)}(z)=-\frac{G'_\mu(z)}{G_\mu(z)}.$$}
\fi

Borrowing tools from the theory of second order freeness \cite{collins2006second}, we use Theorem \ref{thm:3} to prove the following. 

\begin{theorem}[Infinitesimal distributions]
\label{thm:infinitesimaldists}
Let $\mu$ be a probability measure on $\mathbb{R}$ with compact support and suppose that there is a sequence of monic real-rooted polynomials $(p_d)_{d= 1}^\infty$, where $p_d$ is of degree $d$, such that $\kappa_n^d(p_d)$ coincides with the $n$-th free cumulant of $\mu$ for all $n\leq d$. 

Then, the empirical root distribution of $(p_d)_{d=1}^\infty$ has an infinitesimal asymptotic distribution $(\mu,\mu')$ with  infinitesimal Cauchy transform given by
\begin{equation}
\label{eq:Ginfpolynomials}
G_{inf}(z)=\frac{G_{\mu}''(z)}{2G_{\mu}'(z)}-\frac{G_\mu'(z)}{G_{\mu}(z)}.
\end{equation}
%Consequently, \textcolor{blue}{if $M(\mu)$ is a probability measure, then $M(M(\mu))$ is a signed measure and} \begin{equation*} \mu'=\frac{1}{2}\left(M(\mu)-M(M(\mu))\right).\end{equation*}
\end{theorem}

\begin{remark}[Relation to the Markov transform]
\label{rem:markovtransform}
Interestingly, Theorem \ref{thm:infinitesimaldists} can be interpreted in terms of the Markov transform, which hints to potential future research directions. 
Given a probability measure $\mu$ on $\mathbb{R}$, in \cite{kerov1998} Kerov defined the \emph{inverse Markov transform} of $\mu$ as the unique Rayleigh function\footnote{See \cite[Definition 1.4.1]{kerov1998} for a definition of  \emph{Rayleigh function}.} $F:\mathbb{R} \to [0, 1]$  satisfying 
\begin{equation}
\label{eq:defmarkovtransform}
\int_{-\infty}^\infty \frac{1}{z-u}d\mu(u) = \frac{1}{z} \exp \left( -\int_{-\infty}^0\frac{F(u)}{z-u} du+ \int_0^{\infty} \frac{1-F(u)}{z-u} du \right) .
\end{equation}
In general, the derivative of $F$ (in the sense of distributions) is a Schwarz distribution, and therefore one can interpret the inverse Markov transform of $\mu$ as a Schwarz distribution, which we will denote by $M(\mu)$. Moreover, in relevant situations (such as when $\mu$ is a semicircle or a Marchenko-Pastur distribution),   $M(\mu)$ and $M(M(\mu))$ turn out to be compactly supported probability measures themselves, see \cite[Theorem 6.50]{franz2020monotone} for details. In such case (and more in general whenever $G_{M(\mu)}(z)$ and $G_{M(M(\mu))}(z)$ make sense)  (\ref{eq:defmarkovtransform}) yields
$$G_{M(\mu)}(z)=-\frac{G'_\mu(z)}{G_\mu(z)}\quad \text{and} \quad G_{M(M(\mu))}(z)=\frac{G'_\mu(z)}{G_\mu(z)}-\frac{G''_\mu(z)}{G'_\mu(z)},$$
and  $\frac{1}{2}(G_{M(\mu)}(z)-G_{M(M(\mu))}(z))$ turns out to be precisely the right-hand side of (\ref{eq:Ginfpolynomials}). It follows that, in the setting of Theorem \ref{thm:infinitesimaldists}, whenever $M(\mu)$ is a probability measure and $M(M(\mu))$ is a signed measure\footnote{Importantly, we know from \cite[pg. 13]{kerov1998} that the inverse Markov transform preserves the property of being compactly supported. }
\begin{equation}
\label{eq:infinitesimalmarkov}
\mu'=\frac{1}{2}(M(\mu)-M(M(\mu)).
\end{equation}
Finally, we bring to the  attention of the reader that if $p(z)$ is a real-rooted polynomial of degree $d$, and $\nu_0$ and $\nu_1$ are the empirical root distributions of $p(z)$ and $p'(z)$ respectively, then $M(\nu_0) = \nu_0-\nu_1$, which follows since $G_{\nu_0}(z)=\frac{p'(z)}{dp(z)}$ (see \cite[equation (2)]{kerov1998}). We leave as an open direction interpreting this fact in the context of Theorem \ref{thm:infinitesimaldists} and (\ref{eq:infinitesimalmarkov}). 
\end{remark}

Note that Theorem \ref{thm:infinitesimaldists} implies that, under certain assumptions, $G_{inf}$ is explicitly determined by $G_\mu$. Similarly, in Section \ref{sec:infinitesimaldistributions} we show that under the same assumptions, the infinitesimal $R$-transform of $\mu'$ can be explicitly written in terms of the $R$-transform of $\mu$.

Two particularly interesting and important families of polynomials which are included in the above theorem are the Hermite and the Laguerre  polynomials, which appear as the analog of the Gaussian and the Poisson distributions in this theory. Their  empirical distributions converge to the semicircle and Marchenko-Pastur distribution (whose inverse Markov transforms we will compute in Section \ref{sec:examples}), respectively, which are the well known limiting distributions of the GOE/GUE and real/complex Wishart matrices. More interestingly, our formula (\ref{eq:infinitesimalmarkov}) implies that there is a coincidence (up to a sign) between the infinitesimal distributions of Hermite polynomials with the GOE, which was previously shown in \cite{dumitriu2006global} and \cite{kornyik2016wigner} (each with different methods), and Laguerre polynomials with real Wishart matrices, which was also shown in \cite{dumitriu2006global}. In fact, the results in \cite{dumitriu2006global}  apply to all of the so called $\beta$-Hermite and $\beta$-Laguerre ensembles, which for $\beta=\infty$ gives the Hermite and Laguerre polynomials, and for $0<\beta<\infty$ define  random matrix ensembles (in particular, $\beta=1$ yields the GOE and real Wishart ensembles respectively). However, their approach is specific to $\beta$-Hermite and  $\beta$-Laguerre ensembles, and heavily exploits the explicit formulas for the joint density of their eigenvalues. Whereas  the approach presented here, although for the moment inapplicable to random matrices,  is amenable to a wider variety of families of deterministic polynomials (see Section \ref{sec:infinitesimaldistributions}). \\

Apart from this introductory section, the rest of the paper is organized in four other sections. In Section \ref{sec:allpreliminaries} we introduce some notation and survey some of the theory that will be needed throughout this paper. Section \ref{sec:fin.free.mult.conv} is divided in three parts: first we prove Theorem \ref{thm.cumulant.of.products}, that provides a formula for the finite free cumulants of a product of polynomials; then we prove Theorems \ref{thm2} and \ref{thm:weakconvergence}, that articulate in a precise way that the finite free multiplicative convolution converges to the free multiplicative convolution; and the last part retrieves the interesting relation between derivatives of a polynomial and free additive convolution powers. In Section \ref{sec:genusformula} we prove Theorem \ref{thm:3} and its generalization, which gives a topological interpretation of the terms of order $\Theta(1/d^k)$ appearing in Theorem \ref{thm.cumulant.of.products}. The proof of Theorem \ref{thm:infinitesimaldists} and study of the infinitesimal limiting distribution of certain sequences of polynomials is given in Section \ref{sec:infinitesimaldistributions}.

%%%%%%%%%%%%%%%%%%%%%%%%%%%%%%%%%%%%%%%%%%%%%%%%%
%%%%%%%%%%%%%%%%%%%%%%%%%%%%%%%%%%%%%%%%%%%%%%%%%
\section{Preliminaries}
\label{sec:allpreliminaries}

We begin with an introduction of the notation used  throughout this paper regarding set partitions, permutations and polynomials. We then briefly survey some of the theory that will be needed in the sequel. 
\bigskip

\noindent \textbf{Partitions and permutations.} Given a positive integer $n$, a \textit{set partition $\pi$ of $[n]:=\{1,\dots,n\}$} is a set of the form $\pi=\{V_1,\dots,V_k\}$ where the \textit{blocks} $V_1,\dots, V_k\subset [n]$ are pairwise disjoint non-empty subsets of $[n]$ such that $V_1\cup \dots \cup V_k=[n]$. We write $\blocks \pi$ to denote the number of blocks of $\pi$ (in this case $k$).  We denote by $\partlat(n)$ the set of all set partitions of $[n]$. When it is clear from the context we will simply write ``partitions" to refer to set partitions. 

We say that $\pi\in \partlat(n)$ is non-crossing, if for every $1\leq i < j < k < l \leq n$ such that $i, k$ belong to the same block $V$ of $\pi$ and $j,l$ belong to the same block $W$ of $\pi$, then it necessarily follows that all $i,j,k,l$ are in the same block, namely $V=W$. We will denote by $NC(n)$ the set of all non-crossing partitions of $[n]$. We refer the reader to the monograph \cite{nica2006lectures} for a detailed exposition on non-crossing partitions.

As usual, we will turn $\partlat(n)$ into a lattice by equipping it with the reversed refinement order $\leq$, that is, given $\pi, \sigma \in \partlat (n)$ we say that $\pi \leq \sigma$ if every block of $\pi$ is fully contained in a block of $\sigma$. We denote the minimum and maximum in $\partlat(n)$ by $0_n$ and $1_n$ respectively, and use $\pi\vee \sigma$ to denote the supremum of the set $\{\sigma,\pi\}$. 

Let $\mu(\cdot, \cdot)$  be the M\"obius function of the incidence algebra associated to the lattice $\partlat(n)$. It is well known that for any $\pi\in \partlat(n)$ the following formula holds (see \cite[Section 3.10]{stanley2011enumerative})
\begin{equation}
\label{eq:formulaformu}
\mu(0_n, \pi)= (-1)^{n-\blocks{\pi}} \prod_{V\in \pi} (|V|-1)!. 
\end{equation}
We will use $S_n$ to denote the symmetric group on $n$ elements and given a permutation $\alpha\in S_n$ we denote by $\cyc{\alpha}$ the number of cycles of $\alpha$. Every permutation $\alpha\in S_n$ is naturally associated to a partition $f(\alpha)\in P(n)$ with blocks given by the orbits in $[n]$ under the action of $\alpha$. In other words,  $i,j\in [n]$ are in the same block of $f(\alpha)$  if and only if $i$ and $j$ are in the same cycle of $\alpha$. Notice that $\cyc{\alpha}=\blocks{f(\alpha)}$.

Given any sequence $(u_j)_{j=1}^\infty$, and a partition $\pi\in \partlat(n)$ we use the notation 
$$u_{\pi}:=\prod_{V\in \pi} u_{\blocks{(V)}}.$$
Similarly, for $\alpha \in S_n$, we will abuse notation and use $u_\alpha$ as a shorthand notation for $u_{f(\alpha)}$.  
\bigskip

\noindent \textbf{Polynomials.} Let $\monicpols$ be the family of monic polynomials of degree $d$. Given a $p\in\monicpols$ we denote its $d$ roots by $\lambda_1(p), \dots, \lambda_d(p)$ , and for $n=0,1,\dots, d$ denote by $a_n^p$ the $n$-th elementary symmetric polynomials on the roots of $p$, namely
\begin{equation*} 
\label{eq:defi.coeff}
    a_n^p:=\sum_{1\leq i_1<\dots < i_n\leq d} \lambda_{i_1}(p)\cdots \lambda_{i_n}(p),
\end{equation*}
with the convention that $a_0^p:=1$. Recall that since $p$ is monic, then $a_0^p,a_1^p,\dots, a_d^p$ are (up to a sign) the coefficients of $p$, as we have $p(x)=\sum_{i=0}^d x^{d-i} (-1)^i a_i^p$. 

To each polynomial $p\in \monicpols$ we associate its empirical root distribution $$\mu_p := \frac{1}{d} \sum_{i=1}^d \delta_{\lambda_i(p)}.$$
Then, the $n$-th moment of the polynomial $p$, denoted by $m_n(p)$, is defined as the $n$-th moment of $\mu_p$, that is
\begin{equation}
\label{defi.moments}
m_n(p):=\frac{1}{d}\sum_{i=1}^d \lambda_i(p)^n.
\end{equation}
Notice that $dm_n(p)$ is simply the sum of the $n$-th powers of the roots of $p$. Thus, the sequence of moments can be related to the sequence of coefficients using the Newton identities. This gives us a coefficient-moment formula
\begin{equation}
\label{eq.coeff.cumulant}   
a_n^p = \frac{1}{n!} \sum_{\pi \in \partlat (n)} d^{\blocks{\pi}} \mu(0_n, \pi) m_\pi(p), \qquad \text{ for } n=1,2,\dots,d
\end{equation}
This formula can be inverted to write  moments in terms of coefficients, see  \cite[Lemma 4.1]{arizmendi2018cumulants}.

In this paper we will mainly be interested in real-rooted polynomials. So, we will use $\monicpolreal$ to denote the family of real-rooted monic polynomials of degree $d$, and  $\monicpolplus$ to denote the subset of $\monicpolreal$ of polynomials having only non-negative roots.

%%%%%%%%%%%%%%%%%%%%%%%%%%%%%%%%%%%%%%%%%%%%%%%%%%
\subsection{Free probability}
\label{sec:freeprobability}

Here we review some basics of free probability from a combinatorial point of view. For complete introductions to free probability we recommend the monographs \cite{voiculescu1992free, nica2006lectures} and \cite{mingo2017free}. 

Free additive and multiplicative convolutions, denoted by $\boxplus$ and $\boxtimes$ respectively, correspond to the sum or product of free random variables, that is, $\mu_a\boxplus\mu_b=\mu_{a+b}$ and $\mu_a\boxtimes\mu_b=\mu_{ab}$ for $a$ and $b$ free random variables. In this paper, rather than using the notion of free independence we will work solely with the additive and multiplicative convolutions, both of which  can be defined in terms of cumulants. 

For any  probability measure $\mu$, we denote by $m_n(\mu):=\int t^n \mu(dt)$ its $n$-th moment. The \emph{free cumulants} \cite{speicher94multiplicative} of $\mu$, denoted by $(\kappa_n(\mu))_{n=1}^\infty$, are recursively defined via the moment-cumulant formula
\begin{equation*}\label{MCF}
m_n(\mu) =\sum_{\pi\in \nc(n)}\freec_{\pi}(\mu).
\end{equation*} 
It is easy to see that the sequence $(m_n(\mu))_{n=1}^\infty$ fully determines $(\kappa_n(\mu))_{n=1}^\infty$ and vice-versa. So, we can define convolutions of compactly supported measures on the real line via their free cumulants. 

\begin{definition}[Free additive convolution]
Given two compactly supported probability measures $\mu$ and $\nu$ on the real line, we define $\mu\boxplus \nu$ to be the unique measure with cumulant sequence given by
$$\kappa_n(\mu \boxplus \nu) = \kappa_n(\mu)+\kappa_n(\nu).$$
\end{definition}
That $\mu\boxplus \nu$ is a positive measure (in fact compactly supported probability measures on $\R$) follows from \cite{speicher94multiplicative}. 

Let  $\mu^{\boxplus m} =\mu \boxplus \cdots \boxplus \mu$  be the free convolution of $m$ copies of $\mu$. From the above definition, it is clear that $\kappa_n(\mu^{\boxplus m}) = m\kappa_n(\mu)$. In  \cite{nica1996multiplication} Nica and Speicher discovered that one can extend this definition to non-integer powers, we refer the reader to \cite[Section 1]{shlyakhtenko2020fractional} for a discussion on  fractional powers. 

\begin{definition}[Fractional free convolution powers]
Let $\mu$ be a compactly supported probability measure on the real line. For $t\geq 1$, the \emph{fractional convolution power} $\mu^{\boxplus t}$ is defined to be the unique measure with cumulants
$$\kappa_n(\mu^{\boxplus t}) = t \kappa_n(\mu). $$
\end{definition}

It follows from \cite{nica1996multiplication} that  $\mu^{\boxplus t}$ is always well defined and that it is  a compactly supported probability measure on $\R$. Also from \cite{nica1996multiplication} we know that the multiplicative convolution can be defined via cumulants  and that the resulting measure is compactly supported on the real line, when one measure is supported on $[0, \infty)$ and the other one is supported on $\R$. 

\begin{definition}[Free multiplicative convolution]
Let $\mu$ and $\nu$ be compactly supported measures on the real line and assume that the support of $\nu$ is contained in $[0, \infty)$. Then, $\mu\boxtimes \nu$ is the unique measure whose cumulants satisfy the following equation for every $n$\footnote{Given $\pi \in \nc(n)$, $Kr(\pi)$ denotes the Kreweras complement of $\pi$. We refer the reader to \cite[Definition 9.21]{nica2006lectures} for an intuitive geometric  definition of the Kreweras complement and to Section \ref{sec:mapsandpermutations} of this manuscript for an algebraic definition in terms of permutations. }
\begin{equation*}
\label{eq.cumulant.product}
    \freec_n(\mu  \boxtimes \nu ) =  \sum_{\pi \in NC(n)} \freec_\pi(\mu)   \freec_{Kr(\pi)}(\nu). \end{equation*}
\end{definition}

We remind the reader of the following identity (see \cite[Section 14]{nica2006lectures})  that will be used in the sequel
\begin{equation}
    \label{eq:freemomencumulantproduct}
    m_n(\mu\boxtimes \nu) = \sum_{\pi \in NC(n)} \freec_\pi(\mu) m_{Kr(\pi)}(\nu). 
\end{equation}

In the analytic approach to free probability, the \emph{Cauchy transform} (also known as Stieltjes transform) and the  \emph{$R$-transform} of $\mu$ play an important role and are given by
$$G_\mu(z) = \sum_{n=0}^\infty m_n(\mu) z^{-n-1} \quad \text{and} \quad R_\mu(z) =\sum_{n=1}^\infty \kappa_n(\mu) z^{n-1}.$$

Alternatively, recall that  for $z$ near infinity, $G_\mu$ has a compositional inverse which we will denote by $K_\mu$, known as the  \emph{$K$-transform}. Then  the $R$-transform and the $K$-transform are related by the following identity 
$$R_\mu(z)=K_\mu(z)-\frac{1}{z}.$$

%In the proof of Theorem \ref{thm:infinitesimaldists} we will also use 
The $F$-transform of $\mu$, is the multiplicative inverse of the Cauchy transform $F_\mu(z):=1/G_\mu(z)$. It can be seen that if $\mu$ is a compactly supported measure, then $F_\mu(z)$ admits an expansion of the form
$$F_\mu(z)=z+ a_0+\sum_{n=1}^\infty b_n z^{-n}.$$

\subsection{Finite free probability}
\label{sec:preliminaries}

In this section we summarize some definitions and results from \cite{marcus2016polynomial, marcus, arizmendi2018cumulants} on the finite free additive and multiplicative convolutions that will be used throughout the paper. 

%First we look carefully at the definition of the finite free additive and multiplicative convolution, and provide equivalent definitions that will be useful throughout the paper.
\bigskip

\noindent \textbf{Finite free polynomial convolutions.} In what follows we denote the falling factorial by $(d)_k:=\frac{d!}{(d-k)!}=d(d-1)\cdots (d-k+1)$. Then, given two polynomials\footnote{It is possible to extend the definition to all polynomials, and formulas in this paper would be the same up to some normalization. Since our applications are only concerned with the roots of polynomials, we will  restrict our analysis to monic polynomials to keep the notation as simple as possible.} $p,q\in \monicpols$ their finite free additive and multiplicative convolutions, denoted by $p\boxplus_d q$ and $p \boxtimes_d q$ respectively, are defined to be the (unique) polynomials in $\monicpols$ with coefficients given by 
\begin{equation} %\label{eq:coeffsum}
\label{eq:coeffmultiplication}
\frac{a_k^{p\boxplus_d q}}{(d)_k}=\sum_{i+j=k}\frac{a_i^p a_j^q}{(d)_i(d)_j} \qquad\text{and}\qquad a_k^{p\boxtimes_d q}=\frac{a_k^p a_k^q k!}{(d)_k}, \qquad \text{ for } k=1,2,\dots,d.
\end{equation}

Alternatively, these convolutions can be defined via differential operators \cite{marcus2016polynomial}. In particular, as noted by Mirabelli in \cite{mirabelli2021hermitian}, if $D$ denotes differentiation with respect to $x$, and $Q$ and $P$ are polynomials such that $p(x) = P(xD)(x-1)^d$ and $q(x) = Q(xD) (x-1)^d$ then
$$
    [p\boxtimes_d q](x) = P(xD) Q(xD) (x-1)^d = P(xD) q(x) = Q(xD) p(x). 
$$
We also refer the reader to \cite{mirabelli2021hermitian} for a discussion of interesting recent results about the multiplicative convolution. 

\iffalse
 Now we briefly discuss the relation between this polynomial convolutions and the free convolution. For this we will restrict ourselves to the real line, and thus to real-rooted polynomials.  Let $\monicpolreal$ be the family of real-rooted monic polynomials of degree $d$, and let $\monicpolplus$ be the subset of this polynomials having only non-negative roots. 
 \fi
 
 A crucial property of the finite convolutions, is that for all $d\in \N$, the set $\monicpolreal$ is closed under $\boxplus_d$, and the set $\monicpolplus$ is closed under $\boxtimes_d$ \cite{marcus2016polynomial}. Moreover, it is known \cite[Section 16, Exercise 2]{marden1966geometry} that if $p\in \monicpolreal$ and $q\in \monicpolplus$ then $p\boxtimes_d q \in \monicpolreal$.

\iffalse
% $$ m_n^p = \frac{(-1)^{n-1}}{d(n-1)!} \sum_{\pi \in \partlat(n)} (-1)^{\blocks{\pi}-1} N!_\pi (\blocks{\pi}-1)! a_\pi^p. $$
As mentioned in the introduction, one concrete connection between the theories of finite free probability and free probability is that the finite free convolutions converge to the respective free convolution as $d$ goes to infinity. In particular, the analog of Theorem \ref{thm:weakconvergence} for the additive convolution was previously proven in \cite{arizmendi2018cumulants}. 

\begin{theorem}
Consider $(p_d)_{d=1}^\infty$ and $(q_d)_{d=1}^\infty$ with $p_d, q_d\in \monicpolreal$ for every $d$. Furtheremore assume that the $\mu_{p_d}$ and $\mu_{q_d}$ converge weakly to compactly supported probability measures $\mu$ and  $\nu$ respectively. Then, the measures $\mu_{p_d\boxplus q_d}$ converge weakly to $\mu\boxplus \nu$. 
\end{theorem}
\fi
\bigskip

\noindent \textbf{Finite free cumulants.} Inspired by Voiculescu's $R$-transform, which linearizes the additive free convolution, Marcus \cite{marcus} defined the finite $R$-transform, which linearizes the finite free additive convolution $\boxplus_d$. Then, inspired by the combinatorial description of the free cumulants \cite{speicher94multiplicative}, which are the coefficients of the $R$-transform, the finite free cumulants were defined in \cite{arizmendi2018cumulants} as the coefficients of the finite $R$-transform and a combinatorial description of them was provided.  In particular, combinatorial formulas that relate the cumulants with the coefficients and with the moments of a polynomial $p$ where obtained. Since here we do not use the finite $R$-transform, we will directly define the finite free cumulants using the coefficient-cumulant formula, see \cite[Remark 3.5]{arizmendi2018cumulants}.

\begin{definition}[Finite free cumulants] 
\label{defi.finite.free.cumulants}
Let $p\in \monicpols$ with coefficients $(a_n^p)_{n=1}^d$\iffalse defined as in \eqref{eq:defi.coeff}\fi,  the finite free cumulants of $p$, denoted by $(\kappa_n^d(p))_{n=1}^d$ are defined via the cumulant-coefficient formula as follows
\begin{equation}
\label{eq:coefftocumulant}
\kappa_n^d(p) := \frac{(-d)^n}{d(n-1)!} \sum_{\pi \in \partlat(n)} (-1)^{\blocks{\pi}} \frac{N!_\pi a_\pi^p (\blocks{\pi}-1)!}{(d)_{\pi}},   \qquad \text{ for } n=1,2,\dots,d.
\end{equation}
Where $N!_\pi := \prod_{V\in \pi} |V|!$. 
\end{definition}

From \cite{arizmendi2018cumulants}, we know that the coefficients of the polynomial can also be written in terms of its cumulants as follows
\begin{equation}
\label{eq:cumulanttocoefficient}
a_n^p = \frac{(d)_n}{d^n n!} \sum_{\pi \in \partlat(n)} d^{\blocks{\pi}} \mu(0_n, \pi) \kappa_\pi^d(p), \qquad \text{ for } n=1,2,\dots,d.
\end{equation}
Since the moments of a polynomial can be recovered from its coefficients, it is natural to wonder if there is a formula relating moments and finite free cumulants. In \cite[Corollary 4.6]{arizmendi2018cumulants} it was shown that the moment-cumulant formula may be written as
\begin{equation}
\label{eq.momentcumulant}    
m_n(p)=\frac{(-1)^{n-1}}{d^{n+1} (n-1)!} \sum_{\substack{ \tau,\sigma\in P(n) \\  \tau\lor\sigma=1_n }} d^{\blocks{\sigma}+\blocks{\tau}}\mu(0,\sigma)\mu(0,\tau) \kappa_\tau^d(p), \qquad \text{ for } n=1,2,\dots,d. 
\end{equation}
For the formula that writes the cumulants in terms of the moments see \cite[Theorem 4.2]{arizmendi2018cumulants}.

Finally we recall that, as one may expect, the finite free cumulants linearize the finite free convolution.
\begin{proposition}[Proposition 3.6 of \cite{arizmendi2018cumulants}]
For any $p,q\in \monicpols$ and $n=1,\dots,d$ it holds that $$\kappa_n^d(p\boxplus_d q)=\kappa_n^d(p)+\kappa_n^d(q).$$
\end{proposition}
\bigskip

\noindent \textbf{Families of polynomials.} Another important aspect of finite free cumulants, is that they provide a link between the finite free world and the (asymptotic) free world. This is illustrated in the following interesting examples that were originally discovered in \cite[Section 6]{marcus} via the finite $R$-transform.

\begin{example}[Power polynomials and Dirac distributions]
\label{exm.constant.pol.defi}
Fix  $a\in \R$, and for each $d\in \N$ consider the polynomial $p_d(x)=(x-a)^d$ in $\monicpols$. It is easy to see that the coefficients of $p$ are $a_k^p=\binom{d}{k}a^k$ and its moments are $m_n(p_d)=a^n$. After using either the moment-cumulant or the coefficient-cumulant formulas, one obtains that $\kappa_1^d(p_d)=a$ and $\kappa_n^d(p_d)=0$ for $n\geq 2$.

On the other hand, $\mu_{p_d}=\delta_a$ is the Dirac measure at $a$ for every $d$. Thus, when  $d\to\infty$ the limiting distribution is trivially $\mu=\delta_a$. What is worth noticing is that the free cumulants of the limiting distribution are given by $\kappa_1(\mu)=a$ and $\kappa_n(\mu)=0$ for $n\geq 2$; thus $\kappa_n^d(p_d)=\kappa_n(\mu)$ for all $1\leq n\leq d$.

A law of large numbers is valid for the finite free additive convolution and the limiting polynomials are precisely the $p_d(x)$ defined here  \cite[Theorem 6.5]{marcus}.
\end{example}

\begin{example}[Hermite polynomials and semicircular distribution]
\label{exm.hermite.pol.defi}
\iffalse The Hermite polynomials are an important family of orthogonal polynomials.\fi The Hermite polynomial of degree $d$ can be defined explicitly as follows
\[
H_{d}(x):=\sum_{k=0}^{\left\lfloor {\frac {d}{2}}\right\rfloor }
(-1)^{k}\frac {(d)_{2k} }{k!2^{k}}x^{d-2k}.
\]
We will work with the rescaled polynomial $\herm_d(x): = d^{d/2}H_d(\sqrt{d} x)$. It can be easily verified  that the coefficients of $\herm_d$ are given by $a_{2k+1}^{\herm_d}=0$ and 
\[
a_{2k}^{\herm_d}= (-1)^{k}\frac {(d)_{2k} }{k!2^{k}d^{k}},
\]
for every $k$ between 0 and $\lfloor d/2 \rfloor$. It is known \cite[Example 6.1]{arizmendi2018cumulants} that the cumulants of the rescaled Hermite polynomials are given by $\kappa_2^d(\herm_d)=1$ and $\kappa_n^d(\herm_d)=0$ for every $n\neq 2$. And it is a well-known fact that the $\herm_d$ are real-rooted and that,  as $d\to \infty$, the root distributions $\mu_{\herm_d}$ converge to the semicircle law. Again, it is worth noting that the free cumulants of the semicircle law, $\mu_{sc}$, are exactly $\kappa_2(\mu_{sc})=1$ and $\kappa_n(\mu_{sc})=0$ for $n\neq 2$, thus $\kappa_n^d(\herm_d)=\kappa_n(\mu_{sc})$ for all $1\leq n\leq d$.

 Hermite polynomials appear as  limits in  the finite free central limit theorem \cite[Theorem 6.7]{marcus}, and thus play the  role of the Gaussian law in classical probability, and of the semicircular law in free probability.
\end{example}

\begin{example}[Laguerre polynomials and free Poisson distribution]
\label{exm.laguerre.pol.defi}
Following \cite[Chapter V]{szego1939orthogonal}, we define the associated (or generalized) Laguerre polynomials of degree $d$ and parameter $\alpha\in \R$ by
\[
L^{(\alpha)}_{d}(x):=\sum_{k=0}^{d}
\frac{(-x)^{k} (d+\alpha)_{d-k}}{k!(d-k)!}.
\]
When $\alpha\geq -1$ it is a well-known fact that the $(L^{(\alpha)}_{d})_{d\in\N}$ form an orthogonal family of polynomials with respect to a measure supported on $[0, \infty)$, and thus they have real non-negative (and distinct) roots. It can also be seen that for $\alpha=-1,-2,\dots,-d$, the polynomial $L^{(\alpha)}_{d}$ is real-rooted, moreover it has $-\alpha$ roots equal to 0, and all the other roots are distinct. However, for $\alpha\in (-\infty,-1)\backslash \{ -2,-3,\dots,-d\}$, the polynomial $L^{(\alpha)}_{d}$ may have non-real roots.% although numerical experiments suggest that it is real-rooted for $\alpha\in(-2,-1)$ {\color{red} (tal vez esto ya se demostro)}. 

In finite free probability it is more insightful to work with the renormalization 
\[
\lag^{(\lambda)}_d(x):=d! (-d)^{-d} L_d^{((\lambda-1)d)}(dx),
\]
where the constant $d! (-d)^{-d}$ is just to make the polynomial monic. Notice that from the defintion it follows that
\[
a_k^{\lag^{(\lambda)}_d}=\frac{(d)_k(d\lambda)_k}{d^k k!} \qquad \text{for }k=1,\dots,d,
\]
and it can be verified that the finite free cumulants are given by $\kappa_n^d(\lag^{(\lambda)}_d)= \lambda$ for all $n=1,2,\dots,d$ and all $\lambda>0$. These cumulants coincide with the free cumulants of a Marchenko-Pastur distribution of parameter $\lambda$, also known as the free Poisson law. 

Hence for $\lambda>0$, under a suitable approximation $d\to \infty$, the limit root distributions $\mu_{\lag^{(\lambda)}_d}$ of the Laguerre polynomials converge to the Marchenko-Pastur distribution of parameter $\lambda$. By suitable, we mean that for each $\lambda$ we require that all polynomials in the sequence $\{\lag^{(\lambda)}_d\}_{d=1}^\infty$ are real-rooted. Recall, that the above discussion on the real roots of $L^{(\alpha)}_{d}$ implies that $\lag^{(\lambda)}_d\in \monicpolplus$ for $d\geq 1$ and $\lambda\in\{\frac{1}{d},\frac{2}{d}, \dots, \frac{d-2}{d}\}\cup [\frac{d-1}{d},\infty)$. Thus, the approximation works as long as we have $\lambda\geq 1$, or $\lambda=r/s$ is rational and $d$ is a multiple of $s$.

On the other hand, there are examples of Laguerre polynomials (outside the set given above) that have complex roots. For instance, in Remark 6.5 of \cite{arizmendi2018cumulants} it was noticed that if $\lambda=1/3$ and we consider dimension $d=4$, then $\lag^{(1/3)}_4=x^4- \frac{4}{3} x^3 + \frac{1}{6} x^2 + \frac{1}{54} x + \frac{5}{2592}$ has two non-real-roots: $-0.0472193 - 0.0656519i, -0.0472193 + 0.0656519i$.

\iffalse
However, we must warn that $\lag^{(\lambda)}_d$ might no be real-rooted for all $\lambda$. On the other hand, when $\nu_\lambda$ is a free Poisson of parameter $0\leq \lambda <1$ then $\nu_\lambda$ may not belong to $\rrad$. For example, in Remark 6.5 of \cite{arizmendi2018cumulants} it was noticed that if $\lambda=1/3$ and we consider dimension $d=4$, then $P_4(\nu)=x^4- \frac{4}{3} x^3 + \frac{1}{6} x^2 + \frac{1}{54} x + \frac{5}{2592}$ has two non real roots: $0.250561, 1.17721, -0.0472193 - 0.0656519i, -0.0472193 + 0.0656519i$. Therefore, $\nu_{\frac{1}{3}}\notin \rrad$. 

Now recall that for $\lambda>1$ fixed, when $d\to \infty$, the root distributions $\mu_{\lag^{(\lambda)}_d}$ of the Laguerre polynomials converge to the Marchenko-Pastur distribution of parameter $\lambda$, which is the same as the free Poisson law. Notice again, that the free cumulants of the free Poisson law $\mu^{(\lambda)}$ of parameter $\lambda$ are all equal to $\lambda$. Thus $\kappa_n^d(\lag^{(\lambda)}_d)=\kappa_n(\mu^{(\lambda)})$ for all $1\leq d\leq n$, and  hence we may also approach the Marchenko-Pastur law using polynomials for any rational $\lambda=m/n$ as long as we restrict to degrees $d$ multiples of $n$.
\fi

\end{example}

%%%%%%%%%%%%%%%%%%%%%%%%%%%%%%%%%%%%%%%%%%%%%%%%%%
\subsection{Annular permutations}
\label{sec:annularpermutations}

Following the presentation in \cite[Section 5.1]{mingo2017free}, for $\p, \q$ positive integers we define the $(\p, \q)$-annulus to be an annulus in the plane with an \emph{outer} and \emph{inner} circle, where the numbers from $1$ to $\p$ have been arranged in clockwise order on the outer circle, while the numbers from $\p+1$ to $\p+\q$
have been arranged in counterclockwise order on the inner circle. Then, a permutation $\alpha \in S_{\p+\q}$ is said to be annular non-crossing if the cycles of $\alpha$ can be drawn in clockwise order on the $(\p, \q)$-annulus in such a way that:
\begin{enumerate}[(i)]
    \item The cycles do not cross. 
    \item Each cycle  encloses a region, between the inner and outer circle, homeomorphic to the disk with boundary oriented clockwise.
    \item At least one cycle  connects the inner and outer circles.
\end{enumerate}

 \begin{figure}[h]
  \centering
\includegraphics[scale=.35]{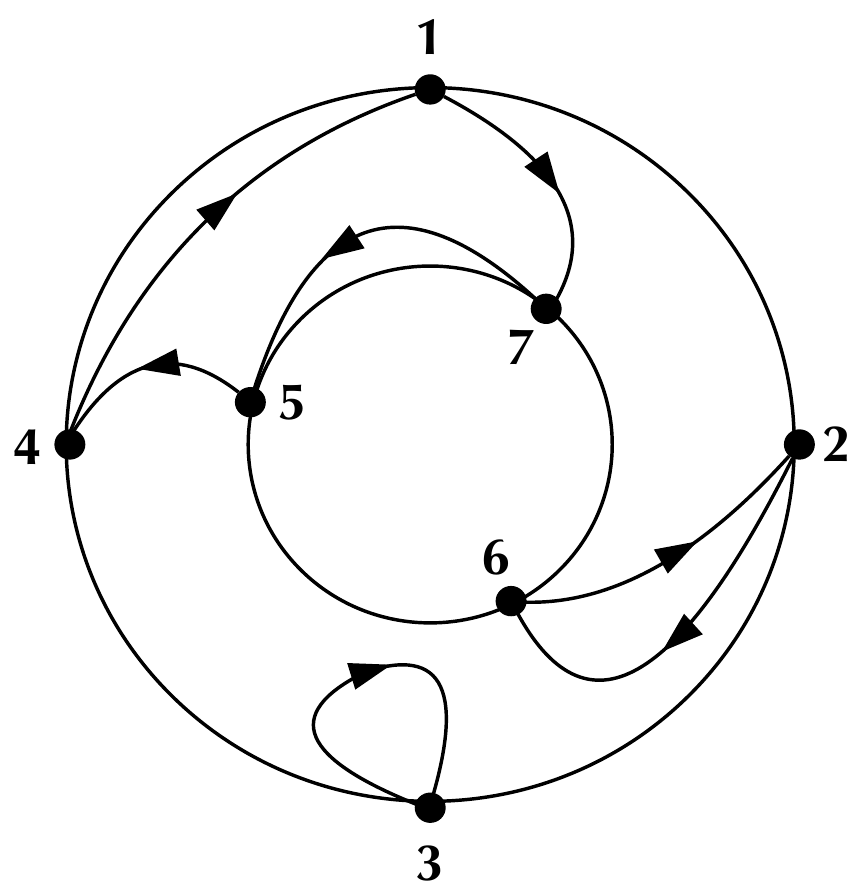}
\caption{Here we show the diagram associated to the permutation $\alpha = (1,7, 5, 4)(3)(2,6)$ in $S_{NC}(4, 3)$. Note that the  ordering  of the elements inside a cycle plays an important role since,  for example, the permutation $\tilde{\alpha} = (1,7,4, 5)(3)(2,6)$ is not in $S_{NC}(4, 3)$ even if $f(\tilde{\alpha})=f(\alpha)$.   }
 \label{fig:annularperm}
\end{figure}

See Figure \ref{fig:annularperm} for an example. We will denote the set of annular non-crossing permutations on the $(\p, \q)$-annulus by $S_{NC}(\p, \q)$.

Annular non-crossing permutations were introduced in \cite{mingo2004annular} to study asymptotic fluctuations of certain random matrix ensembles and are fundamental combinatorial objects in the theory of second order freeness introduced in \cite{mingo2006second}. Although we will not need the notion of second order freeness, to handle sums over $S_{NC}(\p,\q)$ we will use the following functional relation obtained in \cite{collins2006second} between the first and second order Cauchy transforms.  

\begin{lemma}
\label{lem:secondordercauchy}
Let $\mu$ be a measure with free cumulants $(\kappa_n)_{n=1}^\infty$, and let $(\alpha_{\p, \q})_{\p, \q\geq 1}$ be the sequence 
$$\alpha_{\p, \q} := \sum_{\pi \in S_{NC}(\p, \q)} \kappa_\pi. $$
Then the power series $G(z, w) := \frac{1}{zw} \sum_{\p, \q \geq 1} \alpha_{\p, \q} z^{-\p} w^{-\q}$ can be written, at the level of formal power series, in terms of the Cauchy transform of $\mu$ as follows
\begin{equation}
\label{eq:2ndand1stCauchy}
G(z,w)=%\frac{\partial^2}{\partial z  \partial w}\log\left(\frac{G_\mu(w)-G_\mu(z)}{z-w}\right)=
\frac{\partial^2}{\partial z  \partial w}\log\left(\frac{F_\mu(z)-F_\mu(w)}{z-w}\right),\end{equation}
where $F_\mu(z)=1/G_\mu(z)$.

\end{lemma}

\begin{proof}
This follows directly from Theorem 39 in \cite[Section 5]{mingo2017free}%\footnote{Combined with (5.21) from \cite{mingo2017free} that is formally equivalent to (5.19) in the same manuscript, and essentially articulates that we can work with $G$ instead of the $F=1/G$ appearing in Theorem 39.}
, and corresponds to the case when the second order free cumulants are zero,  which is equivalent to the second order $R$-transform being zero.\footnote{Since the notions of second order cumulants and second order $R$-transform are not used in this paper we will not define them here. Instead, we refer the reader to \cite[Section 5]{mingo2017free} for a survey on these and related notions,  and to \cite{collins2006second} for the original source. }
\end{proof}

\begin{remark}
As pointed out in \cite[Section 5]{mingo2017free} taking logarithm in \eqref{eq:2ndand1stCauchy} is well justified since the power series $(F_\mu(z)-F_\mu(w))/(z-w)$ has constant term equal to $1$.
\end{remark}

%We will use series in two variables, say $z$ and $w$, applied when $z=w$. For the convenience of the reader we include the following lemma regarding the difference quotient.

%%%%%%%%%%%%%%%%%%%%%%%%%%%%%%%%%%%%%%%%%%%%%%%%%%

%%%%%%%%%%%%%%%%%%%%%%%%%%%%%%%%%%%%%%%%%%%%%%%%%%
%%%%%%%%%%%%%%%%%%%%%%%%%%%%%%%%%%%%%%%%%%%%%%%%%%
\section{Finite free multiplicative convolution}
\label{sec:fin.free.mult.conv}

\iffalse
Recall, from \cite{marcus2016polynomial}, that the finite free multiplicative convolution of two polynomials $p(x)=\sum_{i=0}^d x^{d-i} (-1)^i a^p_i$ and $q(x)=\sum_{i=0}^d x^{d-i} (-1)^i a^q_i$, with $a^p_0=a^q_0=1$, is given by 
$$(p\boxtimes_d q) (x)=\sum_{i=0}^d x^{d-i} (-1)^i \frac{b_ia_i}{\binom{d}{i}}.$$
\fi

\subsection{Cumulants for multiplicative finite free convolution}

In this section we will prove Theorem \ref{thm.cumulant.of.products}. Let us begin with the proof of the first equation mentioned in the theorem, namely, we will now show that for $p$ and $q$ polynomials of degree $d$, the finite free cumulants of $p\boxtimes_d q$ are given by
\begin{equation}
\label{eq.cumulant.of.products}
\kappa_n^d(p\boxtimes_d q)=\frac{(-1)^{n-1}}{d^{n+1} (n-1)!} \sum_{\substack{ \sigma, \tau\in P(n) \\  \sigma\lor \tau=1_n }} d^{\blocks{\sigma}+\blocks{\tau}}\mu(0_n,\sigma)\mu(0_n,\tau) \kappa_\sigma^d(p) \kappa_\tau^d(q).
\end{equation}

\iffalse
First recall the relation between the coefficients of $p$ and $q$, and the ones of $p\boxtimes_d q$,

\begin{equation}
\label{eq:coeffmultiplication}
a_n^{p\boxtimes_d q}=\frac{a_n^p a_n^q n!}{(d)_n}.
\end{equation}

We will then use the following lemma from \cite{arizmendi2018cumulants} which relates coefficients with finite free cumulants of a polynomials.

\begin{lemma}[Proposition 3.4 in \cite{arizmendi2018cumulants}]
Let $p(x) = \sum_{i=0}^d x^{d-i} (-1)^i a_i^p$ be a polynomial of degree $d$ and let $(\kappa_n^p)_{n=1}^d$ be its finite free cumulants. Then
\begin{equation}
\label{eq:cumulanttocoefficient}
a_n^p = \frac{(d)_n}{d^n n!} \sum_{\pi \in \partlat(n)} d^{\blocks{\pi}} \mu(0_n, \pi) \kappa_\pi^p, \quad n\in \mathbb{N}
\end{equation}
\begin{equation}
\label{eq:coefftocumulant}
\kappa_n^p = \frac{(-d)^n}{d(n-1)!} \sum_{\pi \in \partlat(n)} (-1)^{\blocks{\pi}} \frac{N!_\pi a_\pi^p (\blocks{\pi}-1)!}{(d)_{\pi}}    \end{equation}
\end{lemma}
With these formulas in hand we are now in position to prove Theorem \ref{thm.cumulant.of.products}. 
\fi
In short, the proof of this identity exploits that the coefficients of $p\boxtimes_d q$ can be expressed easily in terms of the coefficients of $p$ and $q$ (recall (\ref{eq:coeffmultiplication})) and that, on the other hand, the cumulants of a polynomial are related to its coefficients by (\ref{eq:coefftocumulant}) and (\ref{eq:cumulanttocoefficient}). 

\begin{proof}[Proof of (\ref{eq.cumulant.of.products})]
Combining (\ref{eq:coeffmultiplication}) with (\ref{eq:coefftocumulant}) we get
\begin{equation}
\label{eq:firststep}
\kappa^d_n(p\boxtimes_d q) =\frac{(-d)^n}{d(n-1)!} \sum_{\pi \in \partlat(n)} (-1)^{\blocks{\pi}} \frac{N!_\pi^2 a_\pi^p a_\pi^q (\blocks{\pi}-1)!}{(d)_{\pi}^2}. 
\end{equation}
Now we will use (\ref{eq:cumulanttocoefficient}) in each of the terms in the above sum. Start by noting that
\begin{align*}
a_\pi^p & = \frac{(d)_{\pi} }{d^n N!_\pi} \prod_{V\in \pi} \left( \sum_{\sigma \in \partlat(|V|)} d^{\blocks{\sigma}} \mu(0_{|V|}, \sigma) \kappa_\sigma^d(p)\right)
\\ & = \frac{(d)_{\pi} }{d^n N!_\pi} \sum_{\sigma \leq \pi} d^{\blocks{\sigma}} \mu(0_n, \sigma) \kappa_\sigma^d(p),    \qquad \qquad \text{since } \mu \text{ is multiplicative.}
\end{align*}
Similarly for $a_\pi^q$. Substituting these formulas we get
\begin{align*}
   \frac{N!_\pi^2 a_\pi^p a_\pi^q}{(d)_{\pi}^2} & = \frac{1}{d^{2n}} \left(\sum_{\sigma \leq \pi} d^{\blocks{\sigma}} \mu(0_n, \sigma) \kappa_\sigma^d(p)  \right) \left( \sum_{\tau \leq \pi} d^{\blocks{\tau}} \mu(0_n, \tau) \kappa_\tau^d(q) \right) 
   \\ & = \frac{1}{d^{2n}} \sum_{\substack{ \tau,\sigma\in P(n) \\  \tau\lor\sigma\leq \pi  }} d^{\blocks{\sigma}+ \blocks{\tau}} \mu(0_n, \sigma) \mu(0_n, \tau) \kappa_\sigma^d(p) \kappa_\tau^d(q).
\end{align*}
Plugging this back into (\ref{eq:firststep}) we get
\begingroup
\allowdisplaybreaks
\begin{align*}
    \kappa^d_n(p\boxtimes_d q) &=\frac{(-1)^n}{d^{n+1}(n-1)!} \sum_{\pi \in \partlat(n)} (-1)^{\blocks{\pi}} (\blocks{\pi}-1)! \sum_{\substack{ \tau,\sigma\in P(n) \\  \tau\lor\sigma\leq \pi  }} d^{\blocks{\sigma}+\blocks{\tau}} \mu(0_n, \sigma) \mu(0_n, \tau) \kappa_\sigma^d(p) \kappa_\tau^d(q)
    \\ & =  \frac{(-1)^n}{d^{n+1} (n-1)!} \sum_{\substack{ \tau,\sigma\in P(n)   }} d^{\blocks{\sigma}+\blocks{\tau}}\mu(0_n,\sigma)\mu(0_n,\tau) \kappa_\sigma^d(p) \kappa_\tau^d(q) \sum_{\pi\geq \sigma \lor \tau } (-1)^{\blocks{\pi}}(\blocks{\pi}-1)!
    \\ &= \frac{(-1)^{n-1}}{d^{n+1} (n-1)!} \sum_{\tau,\sigma\in P(n)} d^{\blocks{\sigma}+\blocks{\tau}}\mu(0_n,\sigma)\mu(0_n,\tau) \kappa_\sigma^d(p) \kappa_\tau^d(q) \sum_{\pi\geq \sigma \lor \tau } \mu(\pi,1_n)
    \\ & = \frac{(-1)^{n-1}}{d^{n+1} (n-1)!} \sum_{\substack{ \tau,\sigma\in P(n) \\  \tau\lor\sigma=1_n }} d^{\blocks{\sigma}+\blocks{\tau}}\mu(0_n,\sigma)\mu(0_n,\tau) \kappa_\sigma^d(p) \kappa_\tau^d(q),
\end{align*}
\endgroup
where the last equality follows from the well known identity $$\sum_{\pi\geq \sigma} \mu(\pi, 1_\pi) = \delta_{\sigma, 1_n}.$$
\end{proof}

\begin{remark}
The above theorem may be proven in a more conceptual way by using the structure of the M\"obius algebra of a lattice. We do this in  Appendix \ref{sec:proofviaMobius}.
\end{remark}

To show the second part of Theorem \ref{thm.cumulant.of.products}, namely the equation  
\begin{equation}
\label{eq:momentofproducts}
m_n(p\boxtimes_d q)=\frac{(-1)^{n-1}}{d^{n+1} (n-1)!} \sum_{\substack{ \sigma, \tau\in P(n) \\  \sigma\lor \tau=1_n }} d^{\blocks{\sigma}+\blocks{\tau}}\mu(0_n,\sigma)\mu(0_n,\tau) \kappa_\sigma^d(p) m_\tau(q),
\end{equation}
we will use Laguerre polynomials to relate the moments of a polynomial to the cumulants of a multiplicative convolution. 

To give some context, recall that in free probability the free Poisson distribution (also called  Marchenko-Pastur)  has free cumulants equal to 1. As explained in Section \ref{sec:preliminaries}, the finite free analog is a normalized Laguerre polynomial, given explicitly by
\begin{equation*}
\label{eq:1poisson}
\lag(x):=\lag^{(1)}(x) = \sum_{i=0}^d (-1)^i \frac{ (d)_i^2}{d^i i!} x^{d-i}, 
\end{equation*}
which satisfies that $\kappa_n^d(\lag)=1$, for $n=1,\ldots,d$. More generally, a free compound Poisson is a random variable whose  sequence of free cumulants equals the sequence of moments of another random variable (see \cite[Section 12]{nica2006lectures}). In this case, the finite free analog was introduced by Marcus in \cite{marcus}. 

Our next lemma, inspired by a well known result in free probability, shows that in the finite free setting, compound Poissons can also be obtained via the multiplicative convolution.

\begin{lemma}[Compound Poissons via multiplication] \label{lem:cumprodpoisson}
Let $p$ be any polynomial and let $\lag(x)$ be as above, then 
$$m_n(p) = \kappa_n^d(p\boxtimes_d \lag),$$
for every $n=1, \dots, d$. 
\end{lemma}

\begin{proof}
As usual, we write $p$ in terms of its coefficients as $p(x) = \sum_{i=0}^d x^{d-i} (-1)^i a^p_i$. Combining (\ref{eq:cumulanttocoefficient}) and (\ref{eq:coeffmultiplication}) we get 
\begin{align*}
\kappa_n^d(p\boxtimes_d \lag) & = \frac{(-d)^n}{d(n-1)!} \sum_{\pi \in \partlat(n)} (-1)^{\blocks{\pi}} \frac{N!_\pi^2 a_\pi^p a_\pi^{\lag}  (\blocks{\pi}-1)!}{(d)_{\pi}^2}
\\ &= \frac{(-1)^n}{d(n-1)!} \sum_{\pi \in \partlat(n)} (-1)^{\blocks{\pi}} N!_\pi a_\pi^p (\blocks{\pi}-1)!
\end{align*}
and the last expression is precisely the right-hand side of the moment-coefficient formula given above in (\ref{eq.momentcumulant}). 
\end{proof}

We can now show the second part of Theorem \ref{thm.cumulant.of.products}.

\begin{proof}[Proof of (\ref{eq:momentofproducts})] Using Lemma \ref{lem:cumprodpoisson} together with \eqref{eq.cumulant.of.products} we get
\begin{eqnarray*}
m_n(p\boxtimes_d q)&=&\kappa_n^d(p\boxtimes_d q \boxtimes_d\lag)=\frac{(-1)^{n-1}}{d^{n+1} (n-1)!} \sum_{\substack{ \sigma, \tau\in P(n) \\  \sigma\lor \tau=1_n }} d^{\blocks{\sigma}+\blocks{\tau}}\mu(0_n,\sigma)\mu(0_n,\tau) \kappa_\sigma^d(p) \kappa_\tau^d(q \boxtimes_d\lag) \\
&=&\frac{(-1)^{n-1}}{d^{n+1} (n-1)!} \sum_{\substack{ \sigma, \tau\in P(n) \\  \sigma\lor \tau=1_n }} d^{\blocks{\sigma}+\blocks{\tau}}\mu(0_n,\sigma)\mu(0_n,\tau) \kappa_\sigma^d(p) m_\tau^d(q).
\end{eqnarray*}
\end{proof}

As mentioned in the introduction,  when $q(x)=(x-1)^d$,  $m_n(q)=1$ for every $n$, so equation (\ref{eq:momentofproducts}) becomes the moment-cumulant formula proven in \cite{arizmendi2018cumulants}. In some sense, this means that the combinatorial  theory of addition can be recovered from that of multiplication. On the other hand, observe that we could have proven Lemma \ref{lem:cumprodpoisson} by combining equation (\ref{eq:momentofproducts}) with the moment-cumulant formula from \cite{arizmendi2018cumulants}, but we have chosen to present our results in the above form to underscore that Lemma \ref{lem:cumprodpoisson} has a simple direct proof and to provide an alternative (and shorter) proof of the existing moment-cumulant formula.

%%%%%%%%%%%%%%%%%%%%%%%%%%%%%%%%%%%%%%%%%%%%%%%%%
\subsection{Finite free multiplication tends to free multiplication}

In this section we will prove Theorems \ref{thm2} and \ref{thm:weakconvergence} by analyzing the asymptotic behavior of the formulas (\ref{eq.cumulant.of.products}) and (\ref{eq:momentofproducts})  as $d\to\infty$. The results appearing in this section articulate in a precise way that the finite free multiplicative convolution converges to the free multiplicative convolution. 
 
Our main tool is the  following lemma which shows how to write a formula that runs over non-crossing partitions as a formula summing over pairs of certain partitions in $\partlat(n)$. This lemma is a particular case of Theorem \ref{thm:generalformula}, which we state and prove in Section \ref{sec:genusformula}, but can also be shown directly using elementary enumerative arguments, as we explain  in Appendix \ref{appendix:alternativeproof}. 

\begin{lemma} \label{Cor:NCtoP}
Let  $(u_j)_{j=1}^\infty$, $(v_j)_{j=1}^\infty$ be two sequences of scalars, then for every $n$ 

\begin{equation} %\label{eq:prop1}
    \sum_{\pi\in NC(n)} u_\pi v_{Kr(\pi)}=\frac{(-1)^{n-1}}{(n-1)!} \sum_{\substack{ \sigma, \tau\in P(n) \\  \sigma\lor \tau=1_n \\ \blocks{\sigma}+\blocks{\tau}=n+1}} \mu(0_n,\sigma)\mu(0_n,\tau) u_\sigma v_\tau.
\end{equation}

\end{lemma}
 
With Theorem \ref{thm.cumulant.of.products} and Lemma \ref{Cor:NCtoP} in hand, the proofs of the expressions
\begin{equation}
\label{eq:asympcumulantcumulant}
\kappa_n^d(p\boxtimes_d q)= \sum_{\pi\in NC(n)} \kappa_\pi^d(p) \kappa_{Kr(\pi)}^d(q)+ O(1/d)
\end{equation}
and 
\begin{equation}
\label{eq:asympmomentcumulant}
m_n(p\boxtimes_d q)= \sum_{\pi\in NC(n)} \kappa_\pi^d(p) m_{Kr(\pi)}(q)+ O(1/d)
\end{equation}
become straightforward. 

\begin{proof}[Proof of Theorem \ref{thm2}] From (\ref{eq.cumulant.of.products})
\begin{align*}
\kappa_n^d(p\boxtimes_d q) &=\frac{(-1)^{n-1}}{d^{n+1} (n-1)!} \sum_{\substack{ \sigma, \tau\in P(n) \\  \sigma\lor \tau=1_n }} d^{\blocks{\sigma}+\blocks{\tau}}\mu(0_n,\sigma)\mu(0_n,\tau) \kappa_\sigma^d(p) \kappa_\tau^d(q) 
\\ &= \frac{(-1)^{n-1}}{ (n-1)!} \sum_{\substack{ \sigma, \tau\in P(n) \\  \sigma\lor \tau=1_n \\ \blocks{\sigma}+\blocks{\tau}=n+1}} \mu(0_n,\sigma)\mu(0_n,\tau) \kappa_\sigma^d(p) \kappa_\tau^d(q) +O(1/d). 
\end{align*}
 So by applying Lemma \ref{Cor:NCtoP} to the above expression we conclude (\ref{eq:asympcumulantcumulant}). Similarly we  obtain (\ref{eq:asympmomentcumulant}) from (\ref{eq:momentofproducts}).
\end{proof}
 
We are now ready to show the following.

\begin{proposition}[Convergence in moments]
\label{prop:convinmoments}
Let $\mu$ and $\nu$ be probability measures with all of their moments finite. Assume that the moments of the sequences $(p_d)_{d=1}^\infty$ and $(q_d)_{d=1}^\infty$ converge to the moments of $\mu$ and $\nu$.  Then the moments of $(p_d\boxtimes_d q_d)_{d=1}^\infty$ converge to the moments of $\mu\boxtimes \nu$.
\end{proposition}

\begin{proof}
As shown in \cite{arizmendi2018cumulants}, from the moment-cumulant formula it follows that the finite free cumulants converge to the free cumulants. So, the assumption that  $\lim_{d\to \infty} m_n(p_d) = m_n(\mu)$ for all $n$, translates into $\lim_{d\to \infty} \kappa_n^d(p_d) = \kappa_n (\mu)$. Similarly we obtain $\lim_{d\to \infty} \kappa_n^d(q_d) = \kappa_n(\nu)$. 

 Then, from (\ref{eq:asympmomentcumulant}) we get that
\begin{align*}
\lim_{d\to \infty} m_n(p_d\boxtimes_d q_d) = \lim_{d\to \infty}   \sum_{\pi \in NC(n)}  \kappa_\pi^d(p_d) m_{Kr(\pi)}(q_d)+O(1/d)
 = \sum_{\pi \in NC(n)} \kappa_\pi(\mu) m_{Kr(\pi)}(\nu).
\end{align*}
So, by the formula for moments of $\mu\boxtimes\nu$  given in (\ref{eq:freemomencumulantproduct}) the right hand side of the above chain of equations is precisely $m_n(\mu \boxtimes \nu)$, and the proof is concluded. 
\end{proof}

Note that the above proposition holds even when $\mu$ and $\nu$ are not supported on the real line and  the polynomials $p_d$ and $q_d$ are not real-rooted. A case of interest, for example, is that of families of polynomials with roots on the unit circle.  

Of course, if the $\mu$ and $\nu$ are compactly supported probability measures on $\mathbb{R}$ and the $p_d$ and $q_d$ are real-rooted, the convergence in moments can be turned into weak convergence of the empirical root distributions  as stated in Theorem \ref{thm:weakconvergence}.

\iffalse
To be clear, if $\lambda_1, \dots, \lambda_d $ are the roots of a polynomial $p(x)$, by the empirical root distribution of $p$ we mean the measure {\color{red} (Daniel: esto está en la sección 2.2, con el nombre $\mu_p$, aunque sólo para polinomios reales, igual puedo extender la definición, y llamarle empirical root distribution para ya sólo referirse a los prelims y poder usar la notación $\mu_p$) } 
$$\frac{1}{d} \sum_{i=1}^d \delta_{\lambda_i}. $$

\begin{theorem}[Weak convergence]
%\label{thm:weakconvergence}
Let $\mu$ and $\nu$ be probability measures supported on a compact subset of the real line. Let $(p_d)_{d=1}^\infty$ and $(q_d)_{d=1}^\infty$ be sequences of polynomials with $p_d\in \monicpolreal$ and $q_d \in \monicpolplus$, and assume that the empirical root distributions of these sequences converge weakly to $\mu$ and $\nu$ respectively. Then,
the empirical root distributions of the sequence $(p_d\boxtimes_d q_d)_{d=1}^\infty$ converge weakly to $\mu\boxtimes \nu$. 
\end{theorem}
\fi

\begin{proof}[Proof of Theorem \ref{thm:weakconvergence}]
If the supports of the empirical root distributions of $p_d$ and $q_d$ are uniformly bounded, then convergence in moments and weak convergence are equivalent, so in this case the theorem follows from Proposition \ref{prop:convinmoments}. 

When the supports of the root distributions are not uniformly bounded one can define auxiliary sequences of \emph{truncated polynomials}, say $(\hat{p}_d)_{d=1}^\infty$ and $(\hat{q}_d)_{d=1}^\infty$, whose roots still converge in distribution to $\mu$ and $\nu$ respectively, but that now have uniformly bounded supports. After noting that, by the above paragraph,  this  implies that the root distributions of $(\hat{p}_d\boxtimes_d \hat{q}_d)_{d=1}^\infty$  converges weakly to $\mu\boxtimes \nu$, one can exploit that $\boxtimes_d$ preserves interlacing of polynomials to argue that the root distributions of $(p_d\boxtimes_d q_d)_{d=1}^\infty$ and $(\hat{p}_d\boxtimes_d \hat{q}_d)_{d=1}^\infty$ have the same weak limit. We refer the reader to Appendix \ref{sec:truncation} for the details on how this is done. 
\end{proof}

%%%%%%%%%%%%%%%%%%%%%%%%%%%%%%%%%%%%%%%%%%%%%%%%%%
%%%%%%%%%%%%%%%%%%%%%%%%%%%%%%%%%%%%%%%%%%%%%%%%%%
\subsection{Derivatives of polynomials tend to free fractional powers}
\label{sec:freefractionalpowers}

In this section we consider a sequence of real-rooted polynomials $(p_d)_{d=1}^\infty$ whose empirical root distributions converge weakly to a compactly supported probability measure $\mu$.  After fixing $t\in (0, 1)$ we will be interested in the asymptotic behavior, as $d\to \infty$, of the root distributions of  $D^{\lfloor (1-t)d \rfloor }p_d(x)$, where $D$ denotes differentiation with respect to $x$. 

Using the PDE characterization of free fractional convolution powers obtained by Shlyakhtenko and Tao \cite{shlyakhtenko2020fractional},  in \cite{steinerberger2019nonlocal} and  \cite{steinerberger2020} Steinerberger informally showed that, under the above setup, the empirical root distributions (after proper normalization) of the polynomials $D^{\lfloor (1-t)d \rfloor } p_d(x) $ converge to $\mu^{\boxplus 1/t}$ as $d\to \infty$. This was later formally proven  by Hoskins and Kabluchko \cite{hoskins2020dynamics} by directly calculating the asymptotics of the $R$-transform of the derivatives of $p_d$.

Our aim is to give an alternative proof which explains this phenomenon from the viewpoint of finite free probability.  In simple terms, we will see that a simple modification of the operator $D^{\lfloor (1-t)d \rfloor}$  may be realized by a finite free multiplicative convolution and then use Theorem \ref{thm:weakconvergence} to obtain the limiting distribution.

\begin{lemma}[Differentiation via multiplication]
\label{lem:diffasmult}
Let $p$ be a monic polynomial of degree $d$. For $q(x)= x^{j} (x-1)^{d-j}$ and $j=0, \dots,d$, we have
$$p(x)\boxtimes_d q(x) = \frac{1}{(d)_j} x^{j} D^{j} p(x).   $$
\end{lemma}

\begin{proof}
First recall from the discussion in Section \ref{sec:preliminaries} that if $Q$ is a polynomial satisfying $q(x) = Q(xD)(x-1)^d$, then 
\begin{equation}
\label{eq:eqeqeq}
p(x)\boxtimes_d q(x) = Q(xD) p(x).
\end{equation}
We begin by showing that, as operators, $x^{i} D^{i}$ is a polynomial in $xD$ for every  $i$. Proceeding by induction assume that $  x^iD ^i =T_i(xD)$ for some polynomial $T_i$ and note that $(xD)   (x^i D ^i)  = ix^i D ^i+ x^{i+1} D^{i+1} $ or equivalently $x^{i+1} D^{i+1}  =  (x D) T_i(xD) - i T_i\left(xD\right),$ proving the existence of $T_{i+1}$.

Now, if we take $Q:=\frac{1}{(d)_j} T_{j}$ it follows that $Q(xD) = \frac{1}{(d)_j} x^{j} D^{j}$. Finally, since clearly $Q(xD) (x-1)^d = \frac{1}{(d)_j} x^{j} D^{j} p(x)= q(x)$, the proof is then concluded by (\ref{eq:eqeqeq}). 
\end{proof}

\begin{remark}
The above lemma can also be proven by first  computing explicitely the coefficients of $x^j(x-1)^{d-j}$ and then using the definition of the $a^{p\boxtimes_d q}_n$ given in (\ref{eq:coeffmultiplication}). However, we believe that the  proof presented above explains at a conceptual level why the operation of taking repeated derivatives can be understood as a multiplicative convolution.    
\end{remark}

The remaining ingredient of our proof is the (by now well known) connection between fractional free convolution powers and the multiplicative free convolution (see \cite[Exercise 14.21]{nica2006lectures}). Namely, if $\mu$ is a compactly supported measure on $\mathbb{R}$, then 
\begin{equation}
\label{eq:freemultandfrac}
(1-t)\delta_0 + t \mu^{\boxplus 1/t} =\Lambda_{1/t}(\mu)\boxtimes \left( (1-t)\delta_0+t \delta_1  \right),
\end{equation}
for any $t\in (0, 1)$, where $\Lambda_{1/t}(\mu)$ denotes the dilation by a factor of $1/t$ of the measure $\mu$.

\iffalse
\begin{proof}
First note that
\[
q(x)=x^{d-k}(x-t)^{k}=x^{d-k}\sum_{j=0}^k \binom{k}{j}  (-t)^j x^{k-j}=\sum_{i=0}^d x^{d-i} (-1)^i b_i,
\]
where $b_i=0$ for $i=k+1,k+2,\dots, d$ and $b_i= \binom{k}{i} t^i$ for $i=0,\dots, k$

Then by the formula for the finite free multiplicative convolution in terms of coefficients \cite{marcus2016polynomial}, we know that

\[
p(x) \boxtimes_d q(x) = \sum_{i=0}^k x^{d-i} (-1)^i a_i \frac{\binom{k}{i} t^i}{\binom{d}{i}}  = \frac{t^d k!(x/t)^{d-k}}{d!} \sum_{i=0}^k (x/t)^{k-i} (-1)^i a_i \frac{(d-i)!}{(k-i)!}
\]
\[
= \frac{t^d k!}{d!}(x/t)^{d-k} p^{(d-k)}(x/t)=\frac{k!}{d!} D_t\Big( x^{d-k
} p^{(d-k)}(x)\Big. 
\]
\end{proof}

In what follows, for a polynomial $p$ of degree $d$, with roots $\lambda_1,\ldots, \lambda_d$, let $$\mu_p=\sum_{i=1}^d \delta_{\lambda_i}$$
denote the empirical distribution of $p$.
\fi

\begin{theorem}[Hoskins-Kabluchko \cite{ hoskins2020dynamics}, Steinerberger \cite{steinerberger2020}]
Fix $t \in (0, 1)$ and let $(p_d)_{d=1}^\infty$ be a sequence of polynomials with $p_d\in \monicpolreal$. Assume that the empirical root distributions of $(p_d)_{d=1}^\infty$ converge weakly to a compactly supported probability measure $\mu$. Then, if  we set $$r_d(x) := D^{\lfloor (1-t )d\rfloor } p_d(t x),$$ the empirical root distributions of $(r_d)_{d=1}^\infty$ will converge weakly to $\mu^{\boxplus 1/t}$. 
\end{theorem}

\begin{proof}
For fixed $d$ define $k:= \lfloor (1-t)d \rfloor$ and $q_d(x):= x^k (x-1)^{d-k}$. Then consider the auxiliary polynomials $\tilde{r}_d(x)$ defined by
$$\tilde{r}_d(x) := p_d(t x) \boxtimes_d q_d(x) = \frac{1}{(d)_k} x^k D^k p_d(t x),$$
where the last equality follows from Lemma \ref{lem:diffasmult}. Now, since $\Lambda_{1/t}(\mu)$ and $(1-t)\delta_0+t \delta_1$ are the asymptotic root distributions of the sequences $(p_d(tx))_{d=1}^\infty$ and $(q_d(x))_{d=1}^\infty$ respectively, by Theorem \ref{thm:weakconvergence} we have that the asymptotic root distribution of $\tilde{r}_d(x)$ is given by
$$\Lambda_{1/t}(\mu) \boxtimes \left( (1-t)\delta_0+t \delta_1  \right),$$
which by (\ref{eq:freemultandfrac}) is precisely the measure $(1-t)\delta_0+t\mu^{\boxplus 1/t}$. The proof is then concluded by noting that $r_d(x)$ has same root distribution as $\tilde{r}_d(x)/x^k$ and that removing the $k$ roots at zero has the effect of removing the atom at 0 from the limiting distribution.  
\end{proof}
%\iffalse

\begin{remark}
Since in the limit, i.e. in free probability, we have the relation $(1-t)\delta_0+t\mu^{\boxplus 1/t}=\Lambda_{1/t}(\mu) \boxtimes \left( (1-t)\delta_0+t \delta_1  \right),$ 
it is natural to compare the analogues of both sides of the equality at the finite level. For simplicity, let us assume that $td$ is an integer. The fractional convolution powers are naturally defined\footnote{The fractional convolution may not have real roots, see \cite[Remark 6.5]{arizmendi2018cumulants}} by their cumulants: $\kappa_n^d(p_d^{\boxplus \lambda})=\lambda \kappa_n^d(p_d)$. The moments for the free powers (normalized by t) are given by
\begin{equation*}
tm_n(p_d^{\boxplus 1/t}) =\frac{(-1)^{n-1}}{d^{n+1} (n-1)!} \sum_{\substack{ \sigma, \tau\in P(n) \\  \sigma\lor \tau=1_n }} d^{\blocks{\sigma}+\blocks{\tau}}\mu(0_n,\sigma)\mu(0_n,\tau) \kappa_\sigma^d(p ) t^{-\blocks{\sigma}+1},
\end{equation*}
while the moments of the free multiplicative convolution of $\tilde p_d(x):=p_d(xt)$ with $q_d(x)$ are given by
\begin{equation*}
m_n(\tilde p_d \boxtimes_d q_d )=\frac{(-1)^{n-1}}{d^{n+1} (n-1)!} \sum_{\substack{ \sigma, \tau\in P(n) \\  \sigma\lor \tau=1_n }} d^{\blocks{\sigma}+\blocks{\tau}}\mu(0_n,\sigma)\mu(0_n,\tau) \kappa_\sigma^d(p ) t^{\blocks{\tau}-n},
\end{equation*}
which in the limit coincides since the leading order is given when $\blocks{\sigma}+\blocks{\tau}=n+1$, i.e. $-\blocks{\sigma}+1=\blocks{\tau}-n$.
\end{remark}

%\fi

\iffalse
Then, if we consider a sequence of polynomials $(p_d)_{d\geq 1}$ whose distribution tends to a probability measure $\mu$, and we define the corresponding sequence $(q_d)_{d\geq 1}$ that tends by construction to $(1-1/t)\delta_0+1/t \delta_t$, then by Theorem ?? we know that 
\[
\mu_{p_d\boxtimes_d q_d}\to \mu  \boxtimes \Big( (1-1/t)\delta_0+1/t \delta_t\Big) 
\]
On the other hand, it is known (see Exercise 14.21 of \cite{nica2006lectures} )
\[
 \mu  \boxtimes \Big( (1-1/t)\delta_0+1/t \delta_t\Big) = (1-1/t)\delta_0+1/t \mu^{\boxplus t}.
\]

Now in terms of measures, Lemma ?? says that
$$\mu_{p\boxtimes _d q}=(1-\frac{1}{t})\delta_0+\frac{1}{t}\mu_{D_t(k!/d! p^{(d-k)})}.$$

Putting all this together we see that 

$$(1-\frac{1}{t})\delta_0+\frac{1}{t}\mu_{D_t(k!/d! p^{(d-k)})}\to(1-\frac{1}{t})\delta_0+\frac{1}{t} \mu^{\boxplus t}.$$

So, excluding the $(1-\frac{1}{t})\delta_0$  we conclude that 
\[
\mu_{D_t(k!/d! p^{(d-k)})}\to \mu^{\boxplus t}
\]
\fi

%%%%%%%%%%%%%%%%%%%%%%%%%%%%%%%%%%%%%%%%%%%%%%%%%%
%%%%%%%%%%%%%%%%%%%%%%%%%%%%%%%%%%%%%%%%%%%%%%%%%%
\section{Genus expansion}
\label{sec:genusformula}
In this section we will give a topological interpretation to the formula  presented in Theorem \ref{thm.cumulant.of.products}. To emphasize that our analysis does not use any particular property pertaining to cumulants or moments, rather than working with pairs of sequences of the form $(\kappa_j^d(p))_{j=1}^d$ and $(\kappa_j^d(q))_{j=1}^d$, or $(\kappa_j^d(p))_{j=1}^d$ and $(m_j(q))_{j=1}^\infty$, we will be working with arbitrary sequences of numbers $(u_j)_{j=1}^\infty$ and $(v_j)_{j=1}^\infty$. 

Recall that the formula in Theorem \ref{thm.cumulant.of.products} has the following form
\begin{equation}
    \label{eq:mainexpression}
\frac{(-1)^{n-1}}{d^{n+1} (n-1)!}    \sum_{\substack{ \tau,\sigma\in P(n) \\  \tau\lor\sigma=1_n }} d^{\blocks{\sigma}+\blocks{\tau}}\mu(0,\sigma)\mu(0,\tau) u_\sigma v_\tau.
\end{equation}
Our starting point is the observation that this sum can be  naturally replaced by a sum over pairs of permutations that generate subgroups with transitive actions. Let us make this precise. 

Recall that $S_n$ denotes the symmetric group on $n$ elements and that $f:S_n \to \partlat(n)$ is the function that transforms permutations into partitions by reading cycles as blocks. \iffalse that is, given $\alpha \in S_n$, $\pi= f(\alpha)$ is the partition of $[n]$  that satisfies $i\sim _\pi j$ if and only if $i$ and $j$ are in the same cycle of $\alpha$\fi From  (\ref{eq:formulaformu}) we note that for any $\pi\in \partlat(n)$,
$$
|\mu(0_n, \pi)|=  \prod_{V\in \pi} (|V|-1)!,
$$
is precisely the number of $\alpha\in S_n$ for which $f(\alpha)=\pi$. Hence, the expression in (\ref{eq:mainexpression}) can be rewritten as 
\begin{equation}
\label{eq:sumoverpermutations}
\frac{(-1)^{n-1}}{d^{n+1}(n-1)!}    \sum_{\substack{ \alpha,\beta\in S_n \\  f(\alpha)\lor f(\beta)=1_n }} (-d)^{\cyc{\alpha}+\cyc{\beta}} u_\alpha v_\beta.
\end{equation}
Moreover, note that the condition  $f(\alpha)\lor f(\beta)=1_n$ is equivalent to the condition that the subgroup generated by $\alpha$ and $\beta$ has a transitive action on $[n]$. With these observations in hand, one can expect to use the general theory of maps and surfaces developed by Jacques \cite{jacques1968genre} and Cori \cite{cori1975code}, to  group the terms in (\ref{eq:sumoverpermutations}) using the notion of genus. We do this  in Theorem \ref{thm:generalformula} below after discussing some preliminaries. 

\subsection{Maps associated to pairs of permutations}
\label{sec:mapsandpermutations}

Here we will review the results and concepts from the theory of maps and surfaces that  will be needed. We refer the reader to \cite[Section 5.1]{mingo2017free} for a more detailed exposition. 

Given two permutations $\alpha, \gamma\in S_n$ we can construct a directed graph, denoted by $\G(\alpha| \gamma)$, whose vertex set is $[n]$ and where there is a directed edge between $i$ and $j$ if $\alpha(i) = j$ or $\gamma(i) =j$, see Figure \ref{fig:orientedgraph}.   Note that, by construction, each vertex in $\G(\alpha|\gamma)$ has two incoming  and two outgoing edges. Moreover, if we ignore the orientation of the edges,  $\G(\alpha|\gamma)$ is connected if an only if the subgroup generated by $\alpha$ and $\gamma$ acts transitively on $[n]$.

 \begin{figure}[h]
  \centering
\includegraphics[scale=.4]{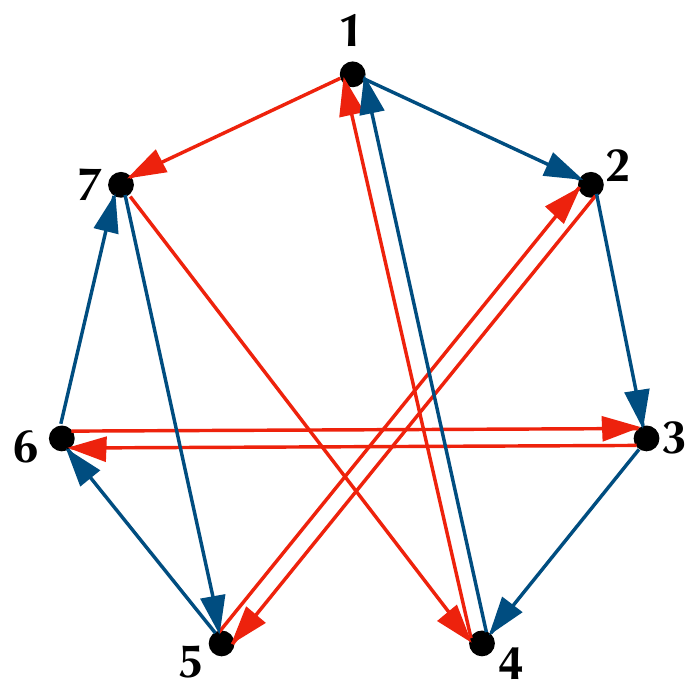}
\caption{  Here we show the oriented graph corresponding to the pair of permutations $\alpha=(1,7,4)(2,5)(3,6)$ and $\gamma=(1,2,3,4)(5,6,7)$. The red and blue edges correspond to the cycles in $\alpha$ and $\gamma$ respectively.    }
 \label{fig:orientedgraph}
\end{figure}

We are interested in embeddings of $\G(\alpha|\gamma)$, into  closed oriented surfaces of genus $g$, with the following properties:
\begin{enumerate}[(i)]
    \item The embedding has no crossings, i.e. the interiors of the edges of $\G(\alpha| \gamma)$ do not intersect.
    \item For every cycle of $\gamma$, the corresponding  cycle of $\G(\alpha| \gamma)$ divides the surface into two regions, such that when transversing the cycle in the direction determined by its directed edges, the region on the \emph{left} is homeomorphic to an open disc and does not contain any vertices or edges of $\G(\alpha| \gamma)$. \item Similarly, for every cycle of $\alpha$, the region on the \emph{right} of the corresponding cycle in $\G(\alpha| \gamma)$ is homeomorphic to an open disc and does not contain any vertices or edges of $\G(\alpha| \gamma)$.
\end{enumerate}
See Figure \ref{fig:annulartorus} for an example. It is known that such  embeddings exist for large enough $g$. The \emph{genus of $\alpha$ relative to $\gamma$} is then defined to be the smallest $g$ for which there exists an embedding with the above properties. 

 \begin{figure}[h]
  \centering
\includegraphics[scale=.4]{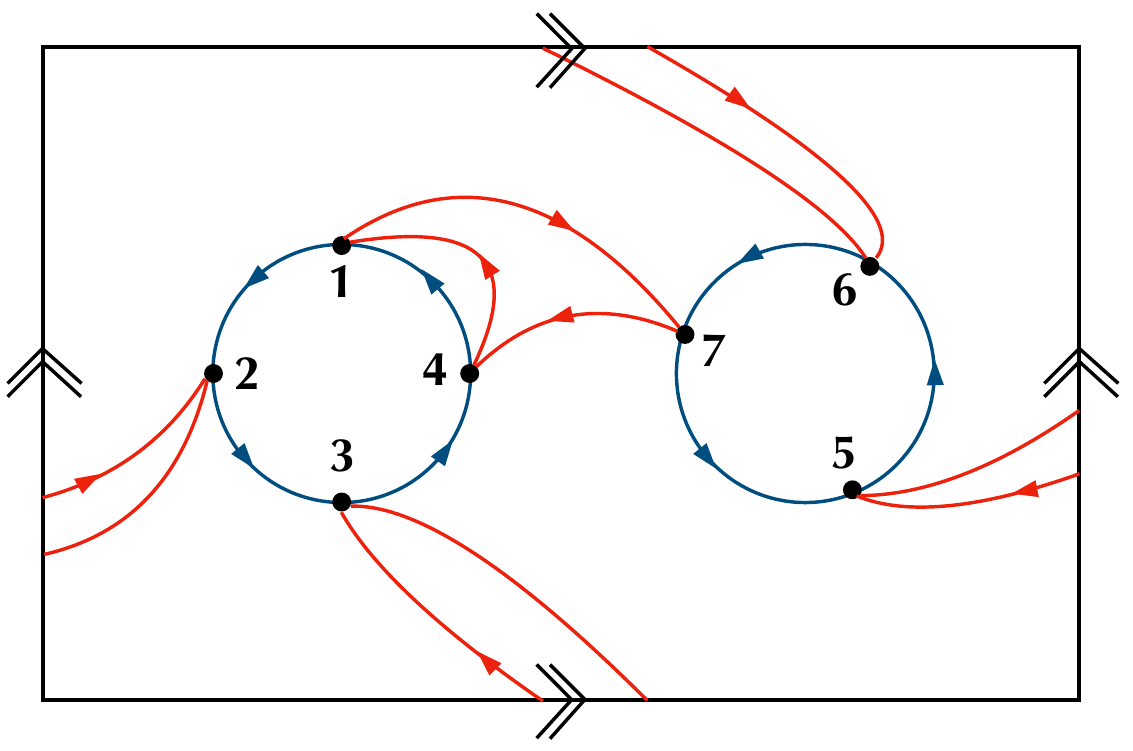}
\caption{A valid embedding into the 2-torus of the directed graph from Figure \ref{fig:orientedgraph}. From the picture it is clear that there is no valid  embedding of this graph into the 2-sphere. }
 \label{fig:annulartorus}
\end{figure}

With this setup, an application of  Euler's formula shows the following classical result that relates the relative genus to the cycle count of the permutations (see Theorem 9 in Section 5.1 of \cite{mingo2017free} for a helpful proof sketch). 

\begin{theorem}
\label{them:JaquesCori}
Let $\alpha, \gamma \in S_n$ and let $g$ be the  genus of $\alpha$ relative to $\gamma$. If the subgroup generated by $\alpha$ and $\gamma$ acts transitively on $[n]$
$$\cyc{\alpha}+\cyc{\alpha^{-1}\gamma} + \cyc{\gamma} = n+2(1-g).$$
\end{theorem}

There are two instances of the above result that are of particular relevance for our discussion. 
\bigskip

\noindent \textbf{Non-crossing partitions.} When $\gamma$ is the $n$-cycle $$\gamma_n := (1, 2, \dots, n),$$ from the properties (i-iii) mentioned above, it is clear that the $\alpha \in S_n$ that have genus 0 with respect to $\gamma_n$ are in one-to-one correspondence with the set of non-crossing partitions of $[n]$.  And, from Theorem \ref{them:JaquesCori},  the set formed by these permutations can be described by
$$S_{NC}(n) := \{\alpha \in S_n : \cyc{\alpha}+\cyc{\alpha^{-1}\gamma_n}=n+1 \}.$$
Moreover, from the graphical representation of $\G(\alpha| \gamma_n)$, it is clear that if $\alpha \in S_{NC}(n)$ then $f(\alpha^{-1}\gamma_n)$ is precisely the Kreweras complement of $f(\alpha)$.

We refer the reader to \cite[Lecture 23]{nica2006lectures} for an extended discussion about the relationship between non-crossing partitions and permutations. 
%this take any $\alpha\in S_n$ of genus 0 relative to $\gamma_n$. Since $g=0$ we know that $\G_{\alpha, \gamma_n}$ can be embedded in the plane without crossings and satisfying the conditions (i-iii) mentioned above. Moreover, if we orient the drawing of the cycle corresponding to $\gamma_n$ in clockwise direction, the drawing of $\alpha$ has to be contained in the bounded region, and it becomes clear that $f(\alpha)$ is non-crossing. 

\bigskip

\noindent \textbf{Annular non-crossing permutations.} Now take two positive integers $\p, \q$ with $\p+\q=n$ and let 
$$\gamma_{\p, \q} :=(1, \dots, \p)(\p+1, \dots, n).$$ 
It is easy to see that the  $\alpha\in S_n$ of genus 0 with respect to $\gamma_{\p, \q}$, that satisfy that the subgroup generated by $\alpha$ and $\gamma_{r, s}$ acts transitively on $[n]$, are precisely the set of non-crossing annular pairings defined above in Section \ref{sec:annularpermutations}. Given $\alpha \in S_{NC}(\p, \q)$ we define its Kreweras complement by $Kr_{\p, \q}(\alpha) := \alpha^{-1}\gamma_{\p, \q}$.
\bigskip

Below we will describe a common generalization of the notions of non-crossing partitions and annular non-crossing permutations.

\subsection{Types and generalized non-crossing permutations}
\label{sec:typesandgeneralizednc}
Here we introduce the notation and remaining concepts that will be used in the proof of the main result of this section. 

Given $\alpha\in S_n$, by the \emph{cycle type} of $\alpha$ we mean  the integer partition $\intpart \vdash n$ that lists the lengths of the cycles of $\alpha$. To be more precise, if $\alpha$ has $m$ cycles, $\intpart$ will be the integer partition $[\intpart_1,   \cdots,  \intpart_m]$, where $\intpart_1\geq \cdots\geq \intpart_m\geq 1$ and such that the $\intpart_i$ are the lengths of the cycles in the cycle decomposition of $\alpha$. For example, if $\alpha \in S_5$ has the cycle decomposition $(1)(2, 4, 5)(3)$, then the cycle type of $\alpha$ is $[3, 1, 1]$.    

Given $\intpart\vdash n$ we will denote the number of parts of $\intpart$ by $\ip{\intpart}$, and  $t_i^\intpart$  will denote the number of parts of $\intpart$ that are equal to $i$. 

If  $\intpart= [\intpart_1, \dots, \intpart_m]$, we will define $\gamma_\intpart$ to be the canonical permutation of type $\intpart$, namely
$$\gamma_\intpart := (1, \dots, \intpart_1) (\intpart_1+1, \dots,\intpart_1+ \intpart_2)\cdots (\intpart_1+\cdots +\intpart_{m-1}+1, \dots, n).$$
To give an example, if $\intpart=[3, 1, 1]$ then $\gamma_\intpart = (123)(4)(5)$. We can now generalize the definitions of $S_{NC}(n)$ and $S_{NC}(\p, \q)$. 

\begin{definition}[Generalized non-crossing permutations]
Given $\intpart \vdash n$ we will denote by $\snc{g}{\intpart}$ to be the set of permutations $\alpha\in S_n$ that have  genus $g$ relative to $\gamma_\intpart$ and  that satisfy that the subgroup generated by $\alpha$ and  $\gamma_\intpart$ acts transitively on $[n]$. 
\end{definition}

For example, if $\alpha = (1, 4, 7)(2, 5)(3, 6)$ then $\alpha\in \snc{1}{4, 3}$, see Figure \ref{fig:annulartorus} for a graphical representation of this permutation as an annular non-crossing permutation on the torus.  Note also that $\snc{0}{n}=S_{NC}(n)$ and  $\snc{0}{\p, \q}=S_{NC}(\p, \q)$. Finally, to simplify many of the expressions below it will be useful to denote
$$S^2_{n, k} := \{ (\alpha, \beta) \in S_n\times S_n : \cyc{\alpha}+\cyc{\beta} = n+1-k  \text{ and }  f(\alpha) \vee f(\beta) = 1_n \}.$$

\subsection{Decomposition by type and genus}
\label{sec:generalformula}

We are now ready to prove the following. 
\begin{theorem}
\label{thm:generalformula}
Let $(u_j)_{j=1}^\infty$ and $(v_j)_{j=1}^\infty$ be any two sequences of numbers. Then, for any $n$ and any $k=0, \dots, n-1$  we have the following decomposition by genus:
\begin{equation}
\label{eq:genusdecomposition}
\frac{(-1)^{n-1}}{(n-1)!} \sum_{\substack{ \sigma, \tau\in P(n) \\ \sigma\lor  \tau=1_n \\ \blocks{\sigma}+\blocks{\tau}=n+1-k}} \mu(0_n,\sigma)\mu(0_n,\tau) u_\sigma v_\tau = (-1)^k \sum_{g=0}^{\lfloor k/2\rfloor} s_k^{(g)},
\end{equation}
where 
\begin{equation}
\label{eq:skg}
s_k^{(g)} =    \sum_{\substack{\intpart \vdash n\\ \ip{\intpart} = k+1-2g}} \frac{n}{\prod_{i=1}^{\ip{\intpart}} \intpart_i \prod_{i=1}^n t_i^\intpart!} \sum_{\substack{ \alpha \in \snc{g}{\intpart}}} u_\alpha v_{\alpha^{-1}\gamma_\intpart}. 
\end{equation}
\end{theorem}

\begin{proof}
 First, from the discussion at the beginning of this section,  we know that we can go from sums over partitions to sums over permutations as follows
$$\frac{(-1)^{n-1}}{(n-1)!} \sum_{\substack{ \sigma, \tau\in P(n) \\  \sigma\lor  \tau=1_n \\ \blocks{\sigma}+\blocks{\tau}=n+1-k}} \mu(0_n,\sigma)\mu(0_n,\tau) u_\sigma v_\tau = \frac{(-1)^k}{(n-1)!} \sum_{\substack{(\alpha, \beta)\in S_{n, k}^2}} u_\alpha v_\beta.$$
Now note that 
 \begin{align*}
 \sum_{(\alpha, \beta)\in S_{n, k}^2} u_\alpha v_\beta  = \sum_{\gamma\in S_n} \sum_{\substack{(\alpha, \beta) \in S_{n, k}^2\\ \alpha\beta=\gamma}} u_{\alpha} v_{\beta}  
  = \sum_{\gamma\in S_n } \sum_{\substack{\alpha \in S_n\\ (\alpha, \alpha^{-1}\gamma )\in S^2_{n, k}}} u_{\alpha} v_{\alpha^{-1} \gamma },  
 \end{align*}
 where the sums corresponding to some of the $\gamma\in S_n$ may be empty. Now we make a crucial observation, if $\gamma, \tilde{\gamma}\in S_n$ have the same type then
 \begin{equation}
 \label{eq:groupingbytype}
 \sum_{\substack{\alpha \in S_n\\ (\alpha, \alpha^{-1}\gamma )\in S^2_{n, k}}} u_{\alpha} v_{\gamma \alpha^{-1}} = \sum_{\substack{\tilde{\alpha} \in S_n\\ (\tilde{\alpha},  \tilde{\alpha}^{-1} \tilde{\gamma})\in S^2_{n, k}}} u_{\tilde{\alpha}} v_{\tilde{\alpha}^{-1}\tilde{\gamma}}.
 \end{equation}
 %Essentially, the above holds because the type assumption implies that there is an inner automorphism of $S_n$ sending $\gamma$ to $\tilde{\gamma}$, and inner automorphisms send transitive subgroups to transitive subgroups  and preserve the type of all permutations. 
 To justify the above equation note that because $\gamma$ and $\tilde{\gamma}$ have the same type we may take $\tau \in S_n$ with $\tau \gamma \tau^{-1} = \tilde{\gamma}$. Then, for any $\alpha \in S_n$ let $\tilde{\alpha}:= \tau \alpha \tau^{-1}$ and note that  the subgroup generated by $ \alpha$ and $\alpha^{-1} \gamma $ acts transitively on $[n]$  if and only if the subgroup generated by $ \tilde{\alpha}$ and $\tilde{\alpha}^{-1} \tilde{\gamma} $ does so too. Moreover $\cyc{\alpha}+\cyc{\alpha^{-1}\gamma}= \cyc{\tilde{\alpha}}+\cyc{\tilde{\alpha}^{-1}\tilde{\gamma} }$, that is, the condition of being in $S^2_{n, k}$ is invariant under conjugation by $\tau$. Since $u_\alpha$ and $v_{\alpha^{-1} \gamma }$ only depend on the type of $\alpha$ and $\alpha^{-1} \gamma $ we also have that $u_\alpha v_{\alpha^{-1}\gamma } = u_{\tilde{\alpha}}v_{\tilde{\alpha}^{-1}\tilde{\gamma} }$, and (\ref{eq:groupingbytype}) follows.  
 
 Now, if for every type $\intpart\vdash n$ we take $\gamma_\intpart$ as a representative of the class of permutations of type $\intpart$, from (\ref{eq:groupingbytype}) we deduce
 \begin{equation}
 \label{eq:takingrepresentative}
  \sum_{\gamma\in S_n } \sum_{\substack{\alpha \in S_n\\ (\alpha, \alpha^{-1}\gamma )\in S^2_{n, k}}} u_{\alpha} v_{\alpha^{-1} \gamma} = \sum_{\intpart\vdash n} N_\intpart  \sum_{\substack{\alpha \in S_n\\ (\alpha, \alpha^{-1} \gamma_\intpart )\in S^2_{n, k}}} u_{\alpha} v_{\alpha^{-1} \gamma_\intpart }  ,  
 \end{equation}
 where $N_\intpart := \frac{n!}{\prod_{i=1}^{|\intpart|}\intpart_i \prod_{i=1}^n t_i^\intpart!}$ is the number of permutations of type $\intpart$. 
 
 Finally, for fixed $\intpart\vdash n$, note that if $\alpha \in S_n$ satisfies that $(\alpha, \alpha^{-1} \gamma_\intpart )\in S_{n, k}^2$ and $g$ is the genus of $\alpha$ relative to $\gamma_\intpart$, then from Theorem \ref{them:JaquesCori} we get
 $$n+2-2g= \cyc{\alpha}+ \cyc{\alpha^{-1} \gamma_\intpart }+\cyc{\gamma_\intpart} = n+1-k+\ip{\intpart},$$
 and hence $g = g_\intpart$ is independent of $\alpha$ where $g_\intpart:= \frac{k+1-\ip{\intpart}}{2}\leq \frac{k}{2}$. Moreover, the subgroup generated by $\alpha$ and $\alpha^{-1} \gamma_\intpart  $ has a transitive action if and only the subgroup generated by $ \alpha$ and $\gamma_\intpart$ does so too. Therefore
  $$\sum_{\substack{\alpha \in S_n\\ (\alpha, \alpha^{-1} \gamma_\intpart )\in S^2_{n, k}}} u_{\alpha} v_{\gamma_\intpart \alpha^{-1}} =  \sum_{\alpha \in \snc{g_\intpart}{\intpart}} u_{\alpha} v_{\alpha^{-1} \gamma_\intpart }.$$
Combining this equation with (\ref{eq:takingrepresentative}) we obtain
  \begin{align*}
  \sum_{\gamma\in S_n } \sum_{\substack{\alpha \in S_n\\ (\alpha, \alpha^{-1} \gamma )\in S^2_{n, k}}} u_{\alpha} v_{\alpha^{-1}\gamma } & = \sum_{\intpart\vdash n} N_\intpart  \sum_{\alpha \in \snc{g_\intpart}{\intpart}} u_{\alpha} v_{\alpha^{-1} \gamma_\intpart }
  \\ & = \sum_{g=0}^{\lfloor k/2\rfloor} \sum_{\substack{\intpart\vdash n \\ \ip{\intpart}= k+1-2g}} N_\intpart  \sum_{\alpha \in \snc{g}{\intpart}} u_{\alpha} v_{\alpha^{-1} \gamma_\intpart}. 
  \end{align*}
  where the last equality is  obtained from grouping terms by the value of $g_\intpart$. The proof is then concluded by combining the above equation with the equalities obtained at the beginning of the proof.  
\end{proof}

We can now obtain Lemma \ref{Cor:NCtoP}  as a corollary of Theorem \ref{thm:generalformula}.

\begin{proof}[Proof of Lemma \ref{Cor:NCtoP}.]  
When $k=0$, $\lfloor k/2\rfloor =0$, so the right hand side of (\ref{eq:genusdecomposition}) has only one term, corresponding to genus 0. Moreover,  the expression in (\ref{eq:skg}) has the unique term $\intpart=[n]$, so the expression simplifies to
$$s_0^{(0)} = \sum_{\alpha\in \snc{0}{n}} u_\alpha v_{\alpha^{-1} \gamma_n }.$$
Now recall, from the discussion in Sections \ref{sec:mapsandpermutations} and  \ref{sec:typesandgeneralizednc}, that $\snc{0}{n}=S_{NC}(n)$ which $f$ bijects with $NC(n)$, and for every $\alpha\in S_{NC}(n)$, $f(\alpha^{-1} \gamma_n )$ is the Kreweras complement of $f(\alpha)$, and the proof is concluded. 
\end{proof}

Similarly, when $k=1$, Theorem \ref{thm:generalformula} reduces to the following lemma.

\begin{lemma}
\label{lem:secondorder}
Let  $(u_j)_{j=1}^\infty$, $(v_j)_{j=1}^\infty$ be two sequences of scalars, then for every $n$ 

\begin{equation} %\label{eq:prop1}
 -  \frac{n}{2} \sum_{\substack{\p+\q=n\\\alpha \in S_{NC}(\p, \q)}} \frac{u_\alpha v_{Kr_{\p, \q}(\alpha)}}{rs}=\frac{(-1)^{n-1}}{(n-1)!} \sum_{\substack{ \sigma, \tau\in P(n) \\  \sigma\lor \tau=1_n \\ \blocks{\sigma}+\blocks{\tau}=n}} \mu(0_n,\sigma)\mu(0_n,\tau) u_\sigma v_\tau.  
\end{equation}

\end{lemma}

\begin{proof}
 Because we are in the case $k=1$, as in the proof of Lemma \ref{Cor:NCtoP} we have $\lfloor k/2\rfloor =0$, so the right hand side of (\ref{eq:genusdecomposition}) again has only one term, namely $-s_1^{(0)}$.  
However, now  the condition $\ip{\intpart}=k+1-2g$ in the expression for $s_1^{(0)}$ simplifies to $\ip{\intpart}=2$. That is, we only care about $\intpart$ of the form $[\p, \q]$ with $\p+\q=n$ and $\p\geq \q$. Hence
$$s_1^{(0)} = \sum_{\substack{\p> \q\geq 1\\ \p+\q=n }} \frac{n}{\p\q} \sum_{\alpha \in \snc{0}{\p, \q}} u_\alpha v_{\alpha^{-1} \gamma_{\p, \q} } + 1\{n\text{ is even}\}\cdot \frac{n}{2\p^2} \sum_{\alpha \in \snc{0}{\p, \p}} u_\alpha v_{\alpha^{-1} \gamma_{\p, \p} } $$
and again by the discussion in Section \ref{sec:typesandgeneralizednc} we know that $\snc{0}{\p, \q} =S_{NC}(\p, \q)$, and by definition $\alpha^{-1}\gamma_{\p, \q} = Kr_{\p,\q}(\alpha)$. 
\end{proof}

The proof of Theorem \ref{thm:3} stated in the introduction is now straightforward. 

\begin{proof}[Proof of Theorem \ref{thm:3}]
  Apply Lemma \ref{lem:secondorder} using the sequences $(u_j)_{j=1}^d = (\kappa_j^d)_{j=1}^d$ and $(v_j)_{j=1}^d = (1)_{j=1}^d$, and combine it with the second equation of Theorem \ref{thm.cumulant.of.products} in the case where $q(x)=(x-1)^d$. 
\end{proof}

%%%%%%%%%%%%%%%%%%%%%%%%%%%%%%%%%%%%%%%%%%%%%%%%%%
%%%%%%%%%%%%%%%%%%%%%%%%%%%%%%%%%%%%%%%%%%%%%%%%%%
\section{Infinitesimal distributions and examples}
\label{sec:infinitesimaldistributions}

In this section we study the infinitesimal limiting distribution of a sequence of polynomials. Recall that a sequence of monic polynomials $(p_d)_{d = 1}^\infty$, where $p_d\in \monicpolreal$, has an asymptotic distribution if 
\[
m_n:=\lim_{d\to \infty} m_n(p_d) <\infty, \qquad \text{for } n\geq 1  
\]
and $(m_n)_{n= 1}^{\infty}$ is the sequence of moments of some distribution $\mu$. Given that such limiting distribution exists, we say that $(p_d)_{d = 1}^\infty$ has an infinitesimal limiting distribution if the following limits also exist
\[
m_n':=\lim_{d\to \infty} d(m_n(p_d)-m_n) <\infty, \qquad \text{for } n\geq 1.
\]
\subsection{Approximation scheme}

The purpose of this section is to provide a general scheme of approximation that allows us to construct sequences of polynomials with infinitesimal limiting distribution. The idea is to fix a compactly supported distribution $\mu$  on the real line and construct monic polynomials $p_d$, of degree $d$, whose finite free cumulants are exactly the first $d$ free cumulants of $\mu$. 

\begin{definition}
Let $\mu$ be a distribution with finite moments of all orders, and denote by $(\freec_n(\mu))_{n= 1}^\infty$ the sequence of free cumulants of $\mu$. 

\begin{enumerate}
    \item 
For every $d\geq 1$ we denote by $p_d(\mu)$ the monic polynomial of degree $d$, uniquely defined by the constraints $\kappa_n^d(p_d(\mu))=\freec_n(\mu)$ for all $1\leq n \leq d$. 
\item Let $S=(d_i)_{i=1}^{\infty}$ be an infinite increasing sequence of natural numbers. A probability measure $\mu$ is said to be an \emph{$S$-real-rooted limit} if $p_{d_i}(\mu)$ is real-rooted for all $d_i\in S$. We denote by $\rrad(S)$ the set of all distributions $\mu$ that are $S$-real-rooted limits.
We use the notation $\rrad=\rrad(\mathbb{N})$.

%\item We say that $\mu$ is a \emph{real-rooted limit} if $p_d(\mu)$ is real-rooted for all $d\geq 1$. We denote by $\rrad$ the set of all distributions $\mu$ that are real-rooted for all degrees.
\end{enumerate}
\end{definition}

From the finite moment-cumulant formula it is clear that 
\[ m_n(\mu)=\lim_{i\to \infty} m_n(p_{d_i}(\mu)). \]
Thus, when $\mu$ is a compactly supported probability measure on the real line, the root distribution of the sequence of polynomials $(p_{d_i}(\mu))_{i=1}^{\infty}$ converges weakly to $\mu$. 
The condition of having real-roots for all degrees $d$ in some subset $S$ is a bit restrictive, however even the more restrictive family, $\rrad$, contains interesting distributions. Also, the families $\rrad(S)$  behave nicely with respect to the free additive convolution and weak convergence. This is summarized in the following proposition.

\begin{proposition}
\label{prop.real.rooted.all.degrees.distributions}
\begin{enumerate}
    \item Let $S_1$ and $S_2$ be sequences such that $S=S_1\cap S_2$ is infinite. If $\mu\in\rrad(S_1)$, and $\nu\in\rrad(S_2)$ then $\mu\boxplus\nu\in\rrad(S)$. In particular, for all $S$, $\rrad(S)$ is closed under the free additive convolution $\boxplus$.

    \item $\rrad(S)$ is closed under convergence in moments. That is, if $(\mu_k)_{k\geq 1} \subset \rrad(S)$, and there exist a probability measure $\mu$ such that $\lim_{k\to \infty} m_n(\mu_k)=m_n(\mu)$ for all $n\geq 1$, then $\mu\in\rrad(S)$.
\end{enumerate}
\end{proposition}

\begin{proof}
\begin{enumerate}

    \item Since $\mu\in\rrad(S_1)$, $\nu\in\rrad(S_2)$ and $S=S_1\cap S_2$, we know that $p_d(\mu)$ and $p_d(\nu)$ are real-rooted for all $d\in S$. Then if we fix a $d\in S$, the finite free convolution $p_d(\mu)\boxplus_d p_d(\nu)$ is also real-rooted. Furthermore, for $1\leq n\leq d$ we have that
    \[
    \kappa_n^d(p_d(\mu)\boxplus_d p_d(\nu))= \kappa_n^d(p_d(\mu))+\kappa_n^d(p_d(\nu))=\kappa_n(\mu)+\kappa_n(\nu)=\kappa_n(\mu\boxplus \nu).
    \]
    So we get that $p_d(\mu\boxplus\nu)=p_d(\mu)\boxplus_d p_d(\nu)$ is real-rooted for all $d\in S$ and we conclude that $\mu\boxplus\nu\in\rrad(S)$.
    \item We start by fixing a $d\in S$. Since convergence in moments implies convergence in free cumulants, we know that 
    \[
    \lim_{k\to \infty} \kappa_n^d(p_d(\mu_k))=\lim_{k\to \infty} \freec_n(\mu_k)=\freec_n(\mu)=\kappa_n^d(p_d(\mu)), \qquad \text{ for } n=1,2,\dots,d.
    \]
    By the coefficient-cumulant formula it follows that the polynomial $p_d(\mu)$ is the limit of the real-rooted polynomials $p_d(\mu_k)$. This implies that $p_d(\mu)$ is real-rooted too. Since this works for all $d\in S$ we conclude that $\mu\in\rrad(S)$.
\end{enumerate}
\end{proof}

\begin{example}  As a direct consequence of Examples \ref{exm.constant.pol.defi}, \ref{exm.hermite.pol.defi} and \ref{exm.laguerre.pol.defi} the following distributions belong to $\rrad$.
\begin{enumerate}
\item The Dirac distribution $\delta_a$ for all $a\in\R$.
\item The semicircular distribution $\mu_{sc}$. 
\item The Marchenko-Pastur distribution of parameter $\lambda\geq 1$.
\end{enumerate}
Moreover, from Example \ref{exm.laguerre.pol.defi}, it can also be seen that for integers $0\leq r\leq s$, the Marchenko-Pastur distribution of parameter $\lambda=r/s$ belongs to $\rrad(s\mathbb{N})$.
\end{example}

\begin{example}[Free compound Poisson distribution] A free compound Poisson distribution of parameter $\lambda$ and jump distribution $\nu$, denoted $\pi(\lambda,\nu)$ is  the probablity measure such that $\freec_n(\pi(\lambda,\nu))=\lambda m_n(\nu)$, and we denote it by $\pi(\lambda,\nu)$.
In our case, we restrict to measures with parameters of specific form. We take $\lambda$ to be a rational number, $\lambda=r/s$, and the jump distribution to be of the form
$$\nu=\frac{1}{n}\sum^n_{i=1}\delta_{\lambda_n}.$$
It can be seen from Example  \ref{exm.laguerre.pol.defi} and Proposition \ref{prop.real.rooted.all.degrees.distributions} that under the above assumption  $$\pi(\lambda,\nu)\in \rrad(sn \mathbb{N}).$$
\end{example}

%\begin{remark}
%{\color{red} (check)} As a direct application of 1 and 2 in the previous proposition we know that $\nu_i \boxplus \dots\boxplus\nu_n\in\rrad$ where for $i=1,\dots,n$ we have that $\nu_i$ is a free Poisson distribution of parameter $\lambda_i\geq 1$. Moreover, by part 3, limits of this type of distributions are also in $\rrad$. For comparison, recall that every compound free Poisson distribution can be approximated by sums of finitely many free Poisson elements of parameters $\lambda_1>0$ (not necessarily bigger than 1), see for instance Exercise 12.25 of \cite{nica2006lectures}. Thus we know %for sure that some 
%that these compound free Poisson distributions are in $\rrad$.
%\end{remark}

The main reason to study distributions $\mu\in\rrad(S)$ is that $(p_d(\mu))_{d\geq 1}$, the canonical sequence  of polynomials associated to $\mu$, has an infinitesimal asymptotic distribution that can be explicitly described in terms of the distribution $\mu$. Depending on whether the information of the distribution $\mu$ is given to us from its moments or its free cumulants, it is sometimes more convenient to use the infinitesimal Cauchy transform or the infinitesimal $R$-transform.

\begin{comment}
\begin{notation}
Given a infinitesimal distribution $\tilde{\mu}=(\mu,\mu')$ with moments $(m_n)_{n\geq 1}$ and infinitesimal moments $(m'_n)_{n\geq 1}$ the Cauchy transform of $\mu$, $G$, and the infinitesimal Cauchy transform, $G_{inf}$, are defined as
\[G(z):=\frac{1}{z}\sum_{n\geq 0} m_n z^{-n} \qquad \text{and}  \qquad G_{inf}(z):=\sum_{n\geq 1} m'_n z^{-n-1}.\]
Let us also consider the compositional inverse of $G$, $K(z):=G^{\langle-1\rangle}(z)$, and the infinitesimal $R$-transform given by
\[R_{inf}(z):=\sum_{n\geq 1} \kappa'_n z^{n-1},\]
where $\kappa'_n$ are the infinitesimal cumulants.
\end{notation}
\end{comment}

%\begin{remark}
%\label{rem.inf.cauchy.R.transform}
The \emph{infinitesimal $R$-transform} of a pair of measures $(\mu,\mu')$, denoted as $R_{inf}(z)$, is a formal series in $z$, that depends on $\mu$ and $\mu'$. As its name suggests, it is the infinitesimal analog of the $R$-transform, due to the fact that it linearizes the infinitesimal free convolution. As we do not work with infinitesimal freness, we will just use the fact from \cite[Theorem 2]{mingo2019non} that the infinitesimal $R$-transform can be written in terms of the $K$-transform (defined in Subsection \ref{sec:freeprobability}) and the infinitesimal Cauchy transform (defined in the introduction) as follows:
\begin{equation}
\label{eq.inf.cauchy.R.transform}
    R_{inf}(z)=-K'_\mu(z)G_{inf}(K_\mu(z)).
\end{equation}
The \emph{infinitesimal free cumulants}, denoted as $(\kappa'_n)_{n \geq 1}^\infty$, are the coefficients of $R_{inf}(z)$.
%\end{remark}

To prove the main result of this section we will use the following lemma regarding the formal power series of $\log(F'(z))$.

\begin{lemma} \label{lemma:diffquotient}
Let $F=F_\mu$ be the $F$- transform of some probability measure $\mu$ with bounded support. 
Then we have the series representation $$\log [F'(z)] =
\sum_{\p,\q\geq 1} \alpha_{\p,\q }\frac{z^{-\p-\q}}{\p\q},$$
where $\alpha_{\p, \q} = \sum_{\pi \in S_{NC}(\p, \q)} \kappa_\pi. $
\end{lemma}

\begin{proof}
We start by analyzing the series in two variables
$$H(z,w):=\frac{F(z)-F(w)}{z-w}.$$

Recall that $F$ is a formal power series of the form $F(z)=z+a_0+\sum^\infty_{n=1}b_nz^{-n}$. Thus, expanding the difference $F(z)-F(w)$ yields
\begin{eqnarray*}
F(z)-F(w)&=&z-w+ \sum^\infty_{n=1}b_n(z^{-n}-w^{-n})
\\&=&z-w+ \sum^\infty_{n=1}b_n(z^{-1}-w^{-1})\Big(\sum_{j+k=n-1}z^{-k}w^{-j}\Big)
\\&=&(z-w)-(z-w)z^{-1}w^{-1}\sum^\infty_{n=1}b_n \sum_{j+k=n-1}z^{-k}w^{-j},
\end{eqnarray*}
from where
\begin{equation}
\label{eq.FminusF}
H(z,w)=1-z^{-1}w^{-1}\sum^\infty_{n=1}b_n \sum_{j+k=n-1}z^{-k}w^{-j}.
\end{equation}

In particular, if we take $w=z$ the one variable series $H(z,z)$ is well defined and is given by
$$H(z,z)=1-\sum^\infty_{n=1}nb_n z^{-n-1}=F'(z).$$

Thus, we are interested in computing the series $\log[H(z,z)]$. Notice that this is almost the content of Lemma \ref{lem:secondordercauchy}, which asserts that
$$\frac{\partial^2}{\partial z  \partial w}\log\left(H(z,w)\right)=\frac{1}{zw} \sum_{\p, \q \geq 1} \alpha_{\p, \q} z^{-\p} w^{-\q}.$$
Formally integrating with respect to $w$ and with respect to $z$, we obtain 
\begin{equation}
\label{eq.logH}
\log\left(H(z,w)\right)=\sum_{\p,\q\geq 1} \alpha_{\p,\q }\frac{z^{-\p}w^{-\q}}{\p\q} + f_1(z) +f_2(w),
\end{equation}
where $f_1(z)$ is the constant (in $w$) term obtained when integrating with respect to $w$ and $f_2(w)$ is the constant (in $z$) term obtained when integrating with respect to $z$. Notice that evaluating $w=z $ yields 
$$\log[F'(z)]=\log\left[H(z,z)\right]=\sum_{\p,\q\geq 1} \alpha_{\p,\q }\frac{z^{-\p-\q}}{\p\q}+f_1(z)+f_2(z).$$

Thus, to conclude the proof we just need to check that $f_1=f_2=0$, which amounts to observing that $\log\left(H(z,w)\right)$ does not have any terms depending only on $z$ or only on $w$. Indeed, from equation \eqref{eq.FminusF} we observe that $H(z,w)=1+z^{-1}w^{-1}J(z,w)$ where $J$ is a series in two variables where all the powers of $w$ and $z$ are non-positive, meaning that $H(z,w)-1$ does not have any term depending only on $z$ or only on $w$. Therefore, when composing with the logarithm series $\log(1+x)=\sum_{n=1}^\infty \frac{(-1)^{n+1}x^n}{n},$ it is now clear that $$\log\left(H(z,w)\right)= \sum_{n=1}^\infty \frac{(-1)^{n+1}}{n}(z^{-1}w^{-1}J(z,w))^n$$
does not contain any terms depending only on $z$ or only on $w$, as desired.
\end{proof}

%With this notation in hand, we can prove the functional equation for $G_{inf}$ advertised in Theorem \ref{thm:infinitesimaldists}, together with an analogous relation for the infinitesimal $R$-transform.
Now, we can prove the functional equation for $G_{inf}$ advertised in Theorem \ref{thm:infinitesimaldists}, together with an analogous relation for the infinitesimal $R$-transform.

\begin{theorem}
\label{thm.rrad}
Let $\mu\in\rrad$. Then, its canonical sequence of monic polynomials, $(p_d(\mu))_{d= 1}^\infty$, has an infinitesimal asymptotic distribution $(\mu,\mu')$. Moreover, the infinitesimal Cauchy transform and the infinitesimal $R$-transform are given by:
\begin{eqnarray}
G_{inf}(z)&=&\frac{G_\mu''(z)}{2G_\mu'(z)}-\frac{G_\mu'(z)}{G_\mu(z)}, \label{eq.G.inf.rrad} \\
R_{inf}(z)&=&\frac{K''_\mu(z)}{2K'_\mu(z)}+\frac{1}{z}. \label{eq.R.inf.rrad}
\end{eqnarray}
\end{theorem}

\begin{proof}
%The main tools to prove this theorem are  Lemma \ref{lem:secondordercauchy} and Theorem \ref{thm:3}. % to compute the infinitesimal distribution of the sequence 

We will combine Lemma \ref{lem:secondordercauchy} and Theorem \ref{thm:3}. to compute the infinitesimal distribution of the sequence  $(p_d(\mu))_{d=1}^\infty$.
First, if $(p_d(\mu))_{d=1}^\infty$ is the canonical sequence
of $\mu\in \rrad$, equation \eqref{eqthm13} from Theorem \ref{thm:3} may be written as
$$m_n(p) - m_n(\mu)= - \, \frac{n}{2 d}  \sum_{ \substack{\p+\q=n\\ \pi \in S_{NC}(\p,\q)}} \frac{\kappa_\pi^d(p)}{\p\q}\, +\,O(1/d^2),$$
which implies that $$m_n' = - \, \frac{n}{2 d}  \sum_{ \substack{\p+\q=n\\ \pi \in S_{NC}(\p,\q)}} \frac{\kappa_\pi^d(p)}{\p\q}\,= - \frac{n}{2}\sum_{\p+\q=n}\frac{1}{\p\q}\alpha_{\p,\q},$$ where we recall the notation from Lemma \ref{lem:secondordercauchy}, $\alpha_{\p, \q} := \sum_{\pi \in S_{NC}(\p, \q)} \kappa_\pi(\mu)$. 

In terms of generating functions we have obtained the relation
\begin{equation} \label{ginf}
G_{inf}(z)=\sum_{n\geq1} m_n'z^{-n-1}=-\frac{1}{2}\sum_{n\geq1} \left(n\sum_{\p+\q=n}\frac{\alpha_{\p,\q}}{\p\q} \right) z^{-n-1}.\end{equation}

% In order to prove that the right-hand side of \eqref{ginf} coincides with the right-hand side of \eqref{eq.G.inf.rrad} we use (\ref{eq:2ndand1stCauchy}) from  Lemma \ref{lem:secondordercauchy}. Indeed, integrating (w.r.t to w and then w.r.t. z) the functional equation (\ref{eq:2ndand1stCauchy}) we get  
%\begin{eqnarray}
%\label{eq:logarithmidentity}
%\log\left(\frac{ F_\mu(w)-F_\mu(z)}{z-w}\right)+f_1(z)+f_2(w)=\sum_{\p,\q\geq 1} \alpha_{\p,\q }\frac{z^{-\p}w^{-\q}}{\p\q} ,\end{eqnarray}
%for some functions $f_1(z)$ and $f_2(w)$. The point here is that when $z=w$ the right-hand side of the above expression  is similar (i.e. equal except for a $-n/2$  in each term) to the expansion of $G_{inf}(z)$ given in (\ref{ginf}).

%Now, to obtain $f_1$ (and similarly for $f_2$), we take the partial derivative with respect to $z$ to deduce that
%\begin{equation*}
%\frac{\partial}{\partial z }\log\left(\frac{F_\mu(w)-F_\mu(z)}{z-w}\right)+f_1'(z)= -\sum_{\p,\q\geq 1} \alpha_{\p,\q }\frac{z^{-\p-1}w^{-\q}}{\q} 
%\end{equation*}
%from which
%$$f_1'(z)=\frac{1}{z-w} - \frac{F_\mu'(z)}{F_\mu(w) - F_\mu(z)}-\sum_{\p,\q\geq 1} \alpha_{\p,\q }\frac{z^{-\p-1}w^{-\q}}{\q}. $$
%Since $F_\mu(w)$ is of the form $F(w)= w+a_0+ a_1w^{-1}+\cdots$, one may deduce that 
%$$f_1'(z)=0~ \quad \left(\text{and similarly }f_2'(w)=0\right).$$
%Thus \eqref{eq:logarithmidentity} is actually written in the following form
Now, from Lemma \ref{lemma:diffquotient} we have that

\begin{eqnarray*}
\log [F_\mu'(z)] &=&
\sum_{\p,\q\geq 1} \alpha_{\p,\q }\frac{z^{-\p-\q}}{\p\q}\\&=& \sum_{n\geq 1}\left(\sum_{\p+\q= n} \frac{\alpha_{\p,\q }}{\p\q} \right) z^{-n}
%\\&=& \log \lim_{w\to z} \frac{ G_\mu(z)-G_\mu(w)}{z-w}+g_1(z)+g_2(z).
\end{eqnarray*}
Finally, by taking a derivative with respect to $z$ and dividing by $2$, we arrive to
$$\frac{1}{2}\frac{\partial}{\partial z}\log [F_\mu'(z)]= -\frac{1}{2}\sum_{n\geq1}
n \left (\sum_{\p+\q=n}\frac{\alpha_{\p,\q}}{\p\q} \right) z^{-n-1}.
$$
Equation \eqref{eq.G.inf.rrad}, now follows from \eqref{ginf} and the fact that 
\begin{eqnarray*}\frac{1}{2} \frac{\partial}{\partial z}\log [F_\mu'(z)]=\frac{1}{2}\frac{F_\mu''(z)}{F_\mu'(z)}=\frac{G_\mu''(z)}{2G_\mu'(z)}-\frac{G_\mu'(z)}{G_\mu(z)}.\end{eqnarray*}

To compute $R_{inf}(z)$ we will use \eqref{eq.inf.cauchy.R.transform}. First, from the equation we just obtained
$$G_{inf}(K_\mu(z))=\frac{G_\mu''(K_\mu(z))}{2G_\mu'(K_\mu(z))}-\frac{G_\mu'(K_\mu(z))}{G_\mu(K_\mu(z))}.$$
Now, differentiating the equation $G_\mu(K_\mu(z))=z$  we have 
$$G_\mu'(K_\mu(z))=\frac{1}{K_\mu'(z)}\quad  \text{and}\quad  G_\mu''(K_\mu(z))=-\frac{K_\mu''(z)}{K_\mu'(z)^3}.$$
Then, by \eqref{eq.inf.cauchy.R.transform} we conclude that
$$R_{inf}(z)=-K_\mu'(z)\left(- \frac{K_\mu''(z)}{2K_\mu'(z)^3}\cdot K_\mu'(z) -\frac{1}{zK_\mu'(z)}\right)=\frac{K_\mu''(z)}{2K_\mu'(z)}+\frac{1}{z},$$
as desired.
\end{proof}
\iffalse
Finally, to conclude the proof of Theorem \ref{thm:infinitesimaldists}, we show that we may rewrite Theorem \ref{thm.rrad} in terms of the inverse Markov transform (see \cite{kerov1998}). Given a probability measure with compact support, the
inverse Markov transform of $\mu$, denoted by $ M(\mu)$, is the unique Schwartz distribution satisfying $$-\frac{\partial}{\partial z}\log(G_\mu(z))=-\frac{G'_\mu(z)}{G_\mu(z)}=G_{M(\mu)}(z).$$
Furthermore, if $M(\mu)$ is a probability measure, we may apply twice the inverse Markov transform and obtain the relation
$$\frac{G'_\mu(z)}{G_\mu(z)}-\frac{G''_\mu(z)}{G'_\mu(z)}=G_{M(M(\mu))}(z),$$
which in our case means that
$
G_{inf}(z)=\frac{1}{2}\left(G_{M(\mu)}(z)-G_{M(M(\mu))}(z)\right).
$
\textcolor{blue}{Now, since weak limits of signed measures are signed measures we have that $\mu'$ is a signed measure, and therefore $G_{inf}(z)$ is the Cauchy transform of a signed measure. Moreover, because the Cauchy transform of sums of signed measures is the sum of the Cauchy transforms of each of the signed measures, we have that  $G_{M(M(\mu))}(z) = 2G_{inf}(z)-G_{M(\mu)}(z)$ is the Cauchy transform of a signed measure, and therefore $M(M(\mu))$ is a signed measure. Moreover, applying the same reasoning, from $
G_{inf}(z)=\frac{1}{2}\left(G_{M(\mu)}(z)-G_{M(M(\mu))}(z)\right)
$ we get 
\begin{equation}\label{eq:Markov} \mu'=\frac{1}{2}\left(M(\mu)-M({M}(\mu))\right).\end{equation}}
\fi

\subsection{Examples}
\label{sec:examples}

\begin{example}[Infinitesimal distribution of Hermite polynomials]
Let $\mu_{sc}$ be the semicircular distribution. As explained in Example \ref{exm.hermite.pol.defi}, we know that in this case the $(p_d(\mu_{sc}))_{d= 1}^\infty$ are precisely the Hermite polynomials $(\herm_d)_{d= 1}^\infty$. It is a well-known fact, see for instance \cite[eq. (2.23)]{nica2006lectures}, that the Cauchy transform of $\mu_{sc}$ is given by $G_{\mu_{sc}}(z)=\frac{z - \sqrt{z^2-4}}{2}$. Hence 
$$G'_{\mu_{sc}}(z) = \frac{\sqrt{z^2-4}-z}{2\sqrt{z^2-4}} \qquad \text{and} \qquad G''_{\mu_{sc}}(z) = \frac{2}{(z^2-4)^{\frac{3}{2}}},$$
so we can use \eqref{eq.G.inf.rrad} to compute
\begin{equation}
G_{inf}(z) 
= \frac{\sqrt{z^2-4}-z}{2(z^2-4)}. \label{eq.G.inf.hermite}
\end{equation} 
Observe that \eqref{eq.G.inf.hermite} can be rewritten as the difference of a symmetric arcsine distribution $\mu_{arc}$ and a Bernoulli distribution $\mu_{ber}$,
$$G_{inf}=\frac{1}{2}\left(\frac{1}{\sqrt{z^2-4}}-\frac{z}{(z-2)(z+2)}\right)=\frac{1}{2}\left(G_{\mu_{arc}}-G_{\mu_{ber}} \right),$$
which is no surprise given formula \eqref{eq:infinitesimalmarkov}, since $\mu_{arc}$ and $\mu_{ber}$ are related to the semicircle distribution via applying once or twice the inverse Markov transform.   Thus the infinitesimal distribution is $$d\mu'(x)=\frac{1}{2}\left(\frac{1}{\pi}\frac{1}{\sqrt{4-x^2}}dx-\frac{\delta_{-2}+\delta_{2}}{2}\right),$$ 
and the infinitesimal moments then are given by $m'_{2n+1}=0$ and
\begin{equation}
\label{eq.inf.moments.hermite}
m'_{2n}=\frac{1}{2}\left(\binom{2n}{n}- 2^{2n} \right), \qquad \forall n\geq 1.
\end{equation} 
This closed formula for the infinitesimal moments can also be obtained directly using the formula for the number of annular pair partitions in $S_{NC}(\p, \q)$ \cite{bousquet2000enumeration}. Notice that infinitesimal moments are (up to a minus sign) given by Sloane's OEIS sequence A000346: $1, 5, 22, 93, 386, 1586, 6476,\dots$.

Finally, from $\eqref{eq.R.inf.rrad}$ since $K_{\mu_{sc}}(z)=\frac{1}{z}+z$ we can compute the infinitesimal $R$-transform $R_{inf}(z) = \frac{z}{z^2-1}$ and from this it is easy to check that the infinitesimal cumulants are $\freec'_{2n+1}=0$ and $\freec'_{2n}=-1$.
\end{example}

\begin{remark}[Hermite polynomials and GOE]
Observe that the infinitesimal Cauchy transform obtained in \eqref{eq.G.inf.hermite} coincides (up to a sign) with the asymptotic infinitesimal distribution of the Gaussian Orthogonal Ensemble, see \cite[Remark 28]{mingo2019non}. Of course this also means that their infinitesimal moments and cumulants coincide (up to a sign), see \cite[Lemma 23 and Corollary 27]{mingo2019non}. Therefore, we just found that the fluctuations of the Hermite polynomials are equal to minus the fluctuations of the GOE, and are given by \eqref{eq.inf.moments.hermite}. This type of result has already appeared in the literature \cite{kornyik2016wigner}. Moreover, the general case for Wigner matrices was solved in  \cite{enriquez2016asymptotic}.

The fact that the $O(1)$ and the $O(1/d)$ terms coincide for the Hermite polynomials and the GOE, prompts the question if this continues to be true for the $O(1/d^2)$ term. However, this turns out to be false and can be observed from small values of $n$. Indeed, from \cite{mingo2019non} and the finite moment-cumulant formula we get the formulas in Table \ref{tab:my_label}. 

\begin{table}[h]
    \centering
    \begin{tabular}{|c|l|l|} \hline
        $n$ & GOE & Hermite polynomial \\ \hline
        $2$ & $1 + d^{-1}$ & $1 - d^{-1}$ \\
        $4$ & $2 + 5d^{-1}+5d^{-2}$ & $2 - 5d^{-1}+3d^{-2}$ \\ 
        $6$ & $5 + 22d^{-1} + 52 d^{-2}+41 d^{-3}$ & $5 - 22d^{-1} + 32 d^{-2}-15 d^{-3}$ \\ \hline
        \end{tabular}
    \caption{Moments of the GOE compared to the moments of the Hermite polynomials.}
    \label{tab:my_label}
\end{table}
\end{remark}

\begin{example}[Infinitesimal distribution of Laguerre polynomials]
Let $\pi(\lambda)$ be the Marchenko-Pastur (free poisson) distribution of parameter $\lambda$. As explained in Example \ref{exm.laguerre.pol.defi}, we know that $(p_d(\pi(\lambda)))_{d\geq 1}$ are the Laguerre polynomials $(\lag_d^{(\lambda)})_{d=1}^\infty$. It is a well-known fact, see for instance \cite[Lecture 12]{nica2006lectures}, that the Cauchy transform of $\pi(\lambda)$ is $$G_{\pi(\lambda)}(z)=\frac{z+1-\lambda +\sqrt{(z-\lambda -1)^2-4\lambda}}{2z}.$$  Again, we can use \eqref{eq.G.inf.rrad} to compute the infinitesimal Cauchy transform which after simplifications has the form
\begin{eqnarray*}
G_{inf}(z) &=& \frac{1}{2}\left(\frac{1}{ \sqrt{(-\lambda + z - 1)^2 - 4 \lambda)}}- \frac{-\lambda + z - 1}{(-\lambda + z - 1)^2 - 4 \lambda}  \right)\label{eq.G.inf.laguerre}
\end{eqnarray*} 
and then we get the infinitesimal distribution,
$$d\mu'(x)=\frac{1}{2}\left(\frac{1}{\pi}\frac{1}{\sqrt{4\lambda-(x-\lambda-1)^2}}dx-\frac{\delta_{(1+\sqrt{\lambda})^2}+\delta_{(1-\sqrt{\lambda})^2}}{2}\right).$$

The moments of $\mu'$ are given then by the difference of the moments of the measures, namely

\begin{equation}
m'_{n}=\frac{1}{2}\left(\sum^n_{k=0}\binom{n}{k}^2\lambda^k-\sum^n_{k=0}\binom{2n}{2k}\lambda^k\right), \qquad \forall n\geq 1.
\end{equation} 
Also since the free cumulants of $\pi(\lambda)$ equal $\lambda$ then $K_{\pi(\lambda)}(z)=\frac{1}{z}+\frac{\lambda}{1-z}$. From \eqref{eq.R.inf.rrad} we can compute the infinitesimal $R$-transform 
$$R_{inf}(z)=\frac{1}{z-1}-\frac{1}{2}\left(
\frac{\sqrt{\lambda}-1}{z-(\sqrt{\lambda}-1)}+\frac{\sqrt{\lambda}+1}{z-(\sqrt{\lambda}+1)}\right),$$
which gives $\kappa'_{n} =1-\frac{1}{2} ((1+\sqrt{\lambda})^n+(1-\sqrt{\lambda})^n)=-\sum^{[n/2]}_{k=1} \binom{n}{2k}\lambda^k$.

For $\lambda=1$ the infinitesimal moments are
\begin{equation}
\label{eq.inf.moments.laguerre}
m'_{n}= \frac{1}{2}\binom{2n}{n}- 4^{n-1}, \qquad \forall n\geq 1.
\end{equation} 
which are given (up to the minus sign) in Sloane's OEIS sequence A008549: 1,6,29,130, 562,2380, 9949, 41226, 169766,$\dots$ %695860 \dots$.
\end{example}

\begin{remark}[Laguerre polynomials and real Wishart matrices]
It is natural to ask if the asymptotic infinitesimal distribution of the Laguerre polynomials coincides with the one of an ensemble of real Wishart matrices. 
The first infinitesimal moments coincide (up to a sign) and are given by 
\begin{eqnarray*}
|m_1'|&=&0, \qquad |m_2'|=\lambda,\qquad |m'_3|=3\lambda + 3\lambda^2, \qquad |m_4'|=6 \lambda+ 17 \lambda^2 + 6 \lambda^3,
 \\ |m_5'|&=&10\lambda + 55\lambda^2 + 55 \lambda^3 + 10 \lambda^4,  \qquad |m_6'|=15 \lambda + 135 \lambda^2 + 262 \lambda^3 + 135 \lambda^4 + 15 \lambda^5,\\
|m_7'|&=&21\lambda + 280 \lambda^2 + 889 \lambda^3 + 889 \lambda^4 + 280 \lambda^5 + 21 \lambda^6.
\end{eqnarray*}
This is true for all $n$, as we learnt from James Mingo, \cite{mingovasquez2021}.

\iffalse
\begin{table}[h]
    \centering
    \begin{tabular}{|c|l|l|} \hline
        $n$ & Wishart matrix & Laguerre polynomial  \\ \hline
        $1$ & $1$ & $1$ \\
        $2$ & $2 + d^{-1}$ & $2 + d^{-1}$ \\ 
        $3$ & $5 + 6d^{-1} + 4 d^{-2}$ & $5 + 6d^{-1} + ?? d^{-2}$ {\color{red}(check)} \\ 
        $4$ & $14 + 29d^{-1} + 42d^{-2}+ 20d^{-3}$ & $14 + 29d^{-1} + ??  d^{-2}+??  d^{-3}$ {\color{red}(check)} \\ \hline
        \end{tabular}
    \caption{Moments of Wishart Matrices compared to the moments of the Laguerre polynomials.}
    \label{tab:my_label}
\end{table}
\fi
\end{remark}

\bigskip

\noindent \textbf{Acknowledgements:} We thank Roland Speicher for leading our attention to the M\"obius algebra and the paper \cite{speicher2000}. We also thank James Mingo, Kamil Szpojankowski and Takahiro Hasebe for various very helpful discussions.

%%%%%%%%%%%%%%%%%%%%%%%%%%%%%%%%%%%%%%%%%%%%%%%
%%%%%%%%%%%%%%%%%%%%%%%%%%%%%%%%%%%%%%%%%%%%%%%%%%

\bibliographystyle{alpha}
\bibliography{FracConv}

\begin{appendix}

\section{Proof of Theorem \ref{thm.cumulant.of.products} via the M\"obius Algebra}
\label{sec:proofviaMobius}

We follow the presentation and notation of \cite{speicher2000}.

A partially ordered set $(P, \leq)$ is called a lattice if each two-element subset, $\{a, b\} \subset P$, has a least upper bound or join, denoted by $a\vee b$, and a  greatest lower bound or meet, denoted by $a\wedge b$.

\begin{notation}
Let $(P, \leq)$ be a finite lattice. 
\begin{enumerate}
    \item For functions $f,g:P\to \mathbb{C}$  we denote by $f*g: P\to \mathbb{C}$ the function defined by
$$f*g~(\pi):=\sum_{\sigma_1\vee\sigma_2=\pi}f(\sigma_1)g(\sigma_2).$$
\item
For a function $f:P\to \mathbb{C}$  we denote by $F(f):P\to \mathbb{C}$ the function defined by
$$F(f) (\pi):=\sum_{\sigma\leq \pi}f(\sigma).$$
\end{enumerate} 
\end{notation}

The main relation between the operation $*$ and the function $F$ is the following.

\begin{proposition}[Proposition 3.3, \cite{speicher2000}] \label{prop: MobiusAlgebra}
Let $(P,\leq)$ be a finite lattice. Then for any functions $f,g:P\to\mathbb{C}$ we have that 
$$F(f*g)=F(f)F(g).$$
\end{proposition}
Now, to prove Theorem 1.1 we consider the lattice $P=P(n)$. We notice that setting $$f(\pi)=d^{\blocks{\pi}}\mu(0,\pi)\kappa^d_{\pi}(p)\qquad\text{and}\qquad g(\pi)=d^{\blocks{\pi}}\mu(0,\pi)\kappa^d_{\pi}(q),$$
we have, by 
\eqref{eq:cumulanttocoefficient}
$$F(f)(1_n)=\frac{n!d^n a^p_n}{(d)_n}\qquad\text{and}\qquad F(g)(1_n)=\frac{n!d^n a^q_n}{(d)_n}.$$
Now, we use Proposition \ref{prop: MobiusAlgebra} to obtain $F(g)F(f)=F(g*f)$,  which means that
$$\frac{(n!)^2d^{2n} a^p_na^q_n}{(d)_n^2}= F(f)(1_n) F(g)(1_n)=F(f*g)(1_n)=\sum_{\pi \in \partlat(n)}b_\pi,$$
where
\begin{equation}\label{baux}
b_n=\sum_{\pi\vee \sigma= 1_n} d^{\blocks{\pi}}\mu(0_n,\pi)\kappa^p_{\pi} d^{\blocks{\sigma}}\mu(0_n,\sigma)\kappa^p_{\sigma}.
\end{equation}
Coming back to $p\boxtimes_d q$, from 
\eqref{eq:coeffmultiplication}  we obtain
$$a^{p\boxtimes_d q}_n=\frac{a^p_na^q_n n!}{ (d)_n}=\frac{(d)_n}{n!d^{2n}} \sum_{\pi \in \partlat(n)}b_\pi=\frac{(d)_n}{n!d^{n}} \left(\frac{1}{d^n} \sum_{\pi \in \partlat(n)}b_\pi\right).$$
Finally, comparing to \eqref{eq:cumulanttocoefficient}
we see that
$$b_n=(-1)^{n-1}d^{n+1}(n-1)!\kappa^d_n(p\boxtimes_d q)$$
and thus from  \eqref{baux}, we arrive to the desired identity.
\section{Truncation via interlacing}
\label{sec:truncation}

In this section we present the details of the truncation argument used in the proof of Theorem \ref{thm:weakconvergence}.  We start by recalling some elementary facts about interlacing polynomials. 

\begin{definition}[Interlacing and common interlacing] Let $p, q \in \monicpolreal$ and denote the roots of $p$ and $q$ by $\lambda_1(p) \leq \cdots \leq \lambda_d(p)$ and $\lambda_1(q)\leq \cdots\leq \lambda_d(q)$ respectively. We say that $q$ interlaces $p$ if
$$\lambda_1(p) \leq \lambda_1(q) \leq \lambda_2(p) \leq \lambda_2(q) \leq \cdots \leq  \lambda_d(p) \leq \lambda_d(q). $$
We say that $p$ and $q$ have a common interlacing if there exists an $r\in \monicpolreal$ that interlaces both $p$ and $q$.
\end{definition}

We will use that convex combinations of polynomials with a common interlacing are real-rooted (see \cite[Lemma 4.5 and proof of Lemma 4.2]{marcus2015interlacing}). 

\begin{lemma}
\label{lem:interlacing}
For $p, q\in \monicpolreal$, $p$ and $q$ have a common interlacing if and only if $tp+(1-t)q\in \monicpolreal$ for every $t\in [0, 1]$. 

Moreover, the $k$-th largest root of $tp+(1-t)q$ lies between the $k$-th largest roots of $p$ and $q$. In particular if $p, q\in \monicpolplus$ and $p$ and $q$ have a common interlacing then $tp+(1-t)q\in \monicpolplus$ for every $t\in [0, 1]$. 
\end{lemma}
 We can then conclude the following.
\begin{lemma}[Preservation of interlacing]
\label{lem:preservinginterlacing}
Take $p, \tilde{p}\in \monicpolreal$ and $q,\tilde{q}\in \monicpolplus$. If $\tilde{p}$ and $p$ have a common interlacing then $\tilde{p} \boxtimes_d q$ and  $p\boxtimes_d q$ have a common interlacing. And, if $\tilde{q}$ and $q$ have a common interlacing then $\tilde{q} \boxtimes_d p$ and  $q\boxtimes_d p$ have a common interlacing.
\end{lemma}

\begin{proof}
Assume that $\tilde{p}$ and $p$ have a common interlacing then, by Lemma \ref{lem:interlacing} $(tp+(1-t)\tilde{p})\in \monicpolreal$. Since multiplicative convolution by $q$ is linear, preserves real-roots and preserves the leading coefficient, then
$$t \tilde{p} \boxtimes_d q + (1-t) p\boxtimes_d q = (tp+(1-t)\tilde{p})\boxtimes_d q\in \monicpolreal \quad  \text{for every }t\in [0, 1].$$
Using Lemma \ref{lem:interlacing} again, this is equivalent to $\tilde{p} \boxtimes_d q $ and $p\boxtimes_d q$ having a common interlacing. 
but because $\tilde{p}$ interlaces $p$ by Lemma \ref{lem:interlacing} we have that $tp+(1-t)\tilde{p} \in \monicpolreal$, and hence  $(tp+(1-t)\tilde{p})\boxtimes_d q\in \monicpolreal$ holds. 

By a similar procedure, if $\tilde{q}$ and $q$ have a common interlacing then, by Lemma \ref{lem:interlacing} $(tq+(1-t)\tilde{q})\in \monicpolplus$. Since multiplicative convolution by $p$ is linear, preserves real-roots and preserves the leading coefficient, then
$$t \tilde{q} \boxtimes_d p + (1-t) q\boxtimes_d p = (tq+(1-t)\tilde{q})\boxtimes_d p\in \monicpolreal \quad  \text{for every }t\in [0, 1].$$
And the conclusion follows in the same way as the previous case.
\end{proof}

\begin{remark}
In a previous version we mistakenly stated Lemma \ref{lem:interlacing} with interlacing (instead of common interlacing). Based on that, the previous version of Lemma \ref{lem:preservinginterlacing} had the following false claim:\\

\emph{Let $p\in \monicpolreal$ and $q,\tilde{q}\in \monicpolplus$. If $\tilde{q}$ and $q$ interlace, then $\tilde{q} \boxtimes_d p$ and  $q\boxtimes_d p$ interlace.}\\

A simple counterexample can be constructed by considering $n=2$ and the polynomials $q(x)=(x-3)(x-0.2)$, $\tilde{q}(x)=(x-.21)(x-3.08621)$ and $p(x)=(x+1)(x-1)$.
\end{remark}

We are ready to present in detail the truncation argument used in the proof of Theorem \ref{thm:weakconvergence}. First recall the setup. We have two sequences $(p_d)_{d=1}^\infty$ and $(q_d)_{d=1}^\infty$ with $p_d \in \monicpolreal$ and $q_d\in \monicpolplus$, with root distributions converging weakly to $\mu$ and $\nu$ respectively. Assume that $M>0$ is such that the supports of $\mu$ and $\nu$ are contained in $[-M, M]$. 

\begin{proof}[Truncation argument in the proof of Theorem \ref{thm:weakconvergence}]
For every $d$ let $n^-(p_d)$ and $n^+(p_d)$ be the number of roots of $p_d$ in $(-\infty, -M]$ and $[M, \infty)$ respectively. Then, define $\hat{p}_d \in \monicpolreal$ to be the polynomial whose $d- n^-(p_d)-n^+(p_d)$ roots in $(-M, M)$ coincide with those roots of $p_d$ that are contained in $(-M, M)$, and set   $n^-(p_d)$ of the remaining roots of $\hat{p}_d$ to be equal to $-M$, while setting the other  $n^+(p_d)$ roots  of $\hat{p}_d$ to be equal to $M$.  Since the empirical root distributions of $p_d$ converge weakly to $\mu$, and $\mu([-M, M])=1$ we conclude that $\lim_{d\to \infty} n^+(p_d)/d = \lim_{d\to \infty} n^-(p_d)/d = 0$. So if  $n_d:= \max \{n^+(p_d), n^-(p_d)\}$ we know that $n_d = o(d)$. 

It follows from the preceding paragraph that the empirical root distributions of $\hat{p}_d$  also converge weakly to $\mu$. Now, for every $d$ we can find a sequence of polynomials $p_d^{(0)}, \cdots , p_d^{(n_d)}$ with $p_d^{(0)}=p_d, p_d^{(n_d)} =\hat{p}_d$ and such that $p_d^{(j+1)}$ and $p_d^{(j)}$ have a common interlacing for every $j=0, \dots, n_{d}-1$. The construction can be done recursively as follows, given a polynomial $p_d^{(i)}$, we let $\lambda_{-}$ be the largest root of $p_d^{(i)}$ in $(-\infty,-M)$ (if any) and let $\lambda_{+}$ be the smallest root of $p_d^{(i)}$ in $(M,\infty)$ (if any). Then $p_d^{(i+1)}$ is constructed with the same roots as $p_d^{(i)}$ except that $\lambda_{-}$ is changed for $-M$ and $\lambda_{+}$ is changed for $M$. Clearly $p_d^{(i)}$ interlaces $p_d^{(i+1)}$ and the values $n^-$, $n^+$ decrease for the new polynomial.

Then, by Lemma \ref{lem:preservinginterlacing}, for every $j=0, \dots, n_d-1$ we have that $p_d^{(j+1)}\boxtimes_d q_d$ and $p_d^{(j)}\boxtimes_d q_d$ have a common interlacing. In turn, this implies that the Kolmogorov distance between the root distributions of $p_d^{(j+1)}$  and $p_d^{(j)}$ (respectively $p_d^{(j)}\boxtimes_d q_d$ and $p_d^{(j+1)}\boxtimes_d q_d$) is less than $2/d$, and hence the Kolmogorov distance between the distributions of $p_d$ and $\hat{p}_d$ (respectively $p_d \boxtimes_d q_d $ and $\hat{p}_d \boxtimes_d q_d$ ) is less than $2n_d/d$. Since $n_d=o(d)$ we conclude that the root distributions of the sequences $(\hat{p}_d)_{d=1}^\infty$ and $(p_d)_{d=1}^\infty$ have the same weak limit (i.e. $\mu$), and similarly for the sequences $(\hat{p}_d\boxtimes q_d)_{d=1}^\infty$ and $(p_d\boxtimes q_d)_{d=1}^\infty$.  

We can define $\hat{q}_d$ in an analogous way and repeat the above arguments to show that the root distributions of $(\hat{q}_d)_{d=1}^\infty$ converge to $\nu$, and that the distributions of $(\hat{p}_d\boxtimes \hat{q}_d)_{d=1}^\infty$ and $(\hat{p}_d\boxtimes q_d)_{d=1}^\infty$ have the same weak limit. With this, the problem has been reduced to studying uniformly bounded families of polynomials. 
\end{proof}

\section{Enumerative proof of Lemma \ref{Cor:NCtoP}}
\label{appendix:alternativeproof}

In this section we will provide an alternative proof for the identity
\begin{equation}
\label{eq:identity}
 \sum_{\pi\in NC(n)} u_\pi v_{Kr(\pi)}=\frac{(-1)^{n-1}}{(n-1)!} \sum_{\substack{ \sigma, \tau\in P(n) \\  \sigma\lor \tau=1_n \\ \blocks{\sigma}+\blocks{\tau}=n+1}} \mu(0_n,\sigma)\mu(0_n,\tau) u_\sigma v_\tau .
\end{equation}

First let us clarify what we mean here by the type of a partition. If $\sigma \in \partlat(n)$ we define the type of $\sigma$ as the $n$-tuple $(s_1, \dots, s_n)$ where $s_i $ is the number of blocks of size $i$ in $\sigma$.\footnote{In the case of permutations we defined the type to be an integer partition rather than an $n$-tuple. Note that both notions are equivalent and only differ at a conceptual level.} Then, the total number of blocks in $\sigma$ equals $s_1+\cdots + s_n$ which we will denote by $\ip{s}$. 

The starting point of our proof is to note that the value of each term $u_\pi v_{Kr(\pi)}$  depends only on the types of $\pi$ and $Kr(\pi)$, and similarly the value of $\mu(0_n,\sigma)\mu(0_n,\tau) u_\sigma v_\tau $  only depends on the types of $\sigma$ and $\tau$. The idea then is that if we group the terms on both sides of (\ref{eq:identity}) by their type, the problem can be reduced to counting pairs of partitions that satisfy certain constraints. To be more precise we need to introduce some notation. 

\begin{definition}
Fix two types $s$  and $t$  with $\ip{s}+\ip{t}=n+1$. 
\begin{enumerate}[i)]
    \item Let $A(s, t)$ be the number of partitions $\pi\in \nc(n)$ of type $s$ such that $Kr(\pi)$ has type $t$. 
    \item Let $B(s, t)$ be the number of pairs $(\sigma, \tau )\in \partlat(n)\times \partlat(n)$ with $\sigma\vee \tau =1_n$ and such that $\sigma$ and $\tau$ have types $s$ and $t$ respectively. 
\end{enumerate}
\end{definition}
Then, the left and right sides of (\ref{eq:identity}) respectively become 
$$ \sum_{s, t} A(s, t) u_s v_t \qquad \text{and} \qquad \frac{(-1)^{n-1}}{(n-1)!} \sum_{s, t} B(s, t) \mu(0_n,s)\mu(0_n,t) u_s v_t,$$
where the sums are over all types $s$ and $t$, and $u_s$ denotes the value of $u_\sigma$ for any $\sigma$ of type $s$ and similarly we define $v_t, \mu(0_n, s)$ and $\mu(0_n, t)$.  So, to show that the above sums are equal, we will show that they coincide term by term, or equivalently we will show that
\begin{equation}
\label{eq:eqeq}
B(s, t) = (-1)^{n-1}\frac{(n-1)!A(s, t)}{\mu(0_n, s)\mu(0_n, t)} 
\end{equation}
for any  $s, t$ with $\ip{s}+\ip{t}=n+1$. 

Now,  from  \cite[Theorem 2.2]{goulden1992combinatorial} (see \cite[Remark 9.24]{nica2006lectures} for a reformulation of the result in our context) we know that
$$A(s, t)= \frac{n(\ip{s}-1)!(\ip{t}-1)!}{(s_1!\cdots s_n!)(t_1!\cdots t_n!)}.$$
Moreover,  we also have formulas for $\mu(0_n, s)$ and $\mu(0_n, t)$ (see (\ref{eq:formulaformu}) above). So, substituting the formulas for $A(s, t)$, $\mu(0_n, s)$ and $\mu(0_n, t)$, we see that proving (\ref{eq:eqeq}) (and therefore showing Lemma \ref{Cor:NCtoP}) boils down to proving the following lemma, which is the main result of this section. 

\begin{lemma}
For any types $s$ and $t$ with $\ip{s}+\ip{t}=n+1$ 
$$B(s, t) = \frac{n!(\ip{s}-1)!(\ip{t}-1)!}{(s_1!\cdots s_n!)(t_1!\cdots t_n!)(2!^{s_3} 3!^{s_4} \cdots(n- 1)!^{s_n})(2!^{t_3} 3!^{t_4} \cdots(n- 1!)^{t_n})}.$$
\end{lemma}

\begin{proof}
The proof consists of two steps. First we will show that $B(s, t) \ip{s}!\ip{t}!$ counts a certain family of labeled trees. Then we will compute the number of such trees. But first let us introduce some notation. 

By a \emph{sorted partition} we mean an element  in $ \partlat(n)$ that has been endowed with a linear ordering of its elements. E.g. $(\{ 1\},\{2\},\{3, 4\} )$ and $( \{2\}, \{1\}, \{3, 4\})$ are distinct sorted partitions that come from the same element in $\partlat(4)$. Let $\mathcal{P}_{s, t}^{\text{or}}$ be the set of pairs of sorted partitions $(\sigma, \tau)$ that satisfy that $\sigma$ and $\tau$ are of types $s$ and $t$ respectively, and $\sigma \vee \tau =1_n$.  Note that by definition
$$|\calP_{s, t}^{\text{or}}| = B(s, t) \ip{s}!\ip{t}!.$$
Let $K_{s, t}$ denote the complete bipartite graph with vertex components of size $\ip{s}$ and $\ip{t}$. Henceforth, we will set the convention that the vertices in the $s$-component of $K_{ s, t}$ are enumerated with the numbers in $\{1, \dots, \ip{s}\}$ and the vertices in the $\ip{t}$-component are enumerated with the numbers in $\{\ip{s}+1, \dots, \ip{s}+\ip{t}\}$. Let $\mathcal{T}_{s, t}$ denote the set of spanning trees of $K_{ s, t}$ whose edges have been labelled using the numbers from $1$ to $n$ without repeating labels (since $\ip{s}+\ip{t}=n+1$ note that each label is used exactly once), and such that the degree sequences of the $s$-component and $t$-component have types $s$ and $t$ respectively.  
\bigskip

\emph{Step 1: Bijecting  $\calP_{s, t}^{\text{or}}$ and $\mathcal{T}_{s, t}$. } We will now describe a reversible procedure that constructs a tree in $\mathcal{T}_{s, t}$ from a pair in $\calP_{s, t}^{\text{or}}$.   Take any  $(\sigma, \tau)\in \calP_{s, t}^{\text{or}}$ and note that since the blocks in $\sigma $ and $\tau$ are sorted, it is valid to say that we assign the $i$-th block in $\sigma$ to the vertex $i$ in the $s$-component of $K_{s, t}$ and the $j$-th block $\tau$  to the vertex $\ip{s}+j$ in its $t$-component. Then, construct an edge-labelled subgraph $T_{(\sigma, \tau)}$ of $K_{s, t}$ as follows: put an edge with label $i$ between $V\in \sigma$ and $W\in \tau$ if and only if $V\cap W =\{i\}$. See Figure \ref{fig:tree} for an example.

 \begin{figure}[h]
  \centering
\includegraphics[scale=.4]{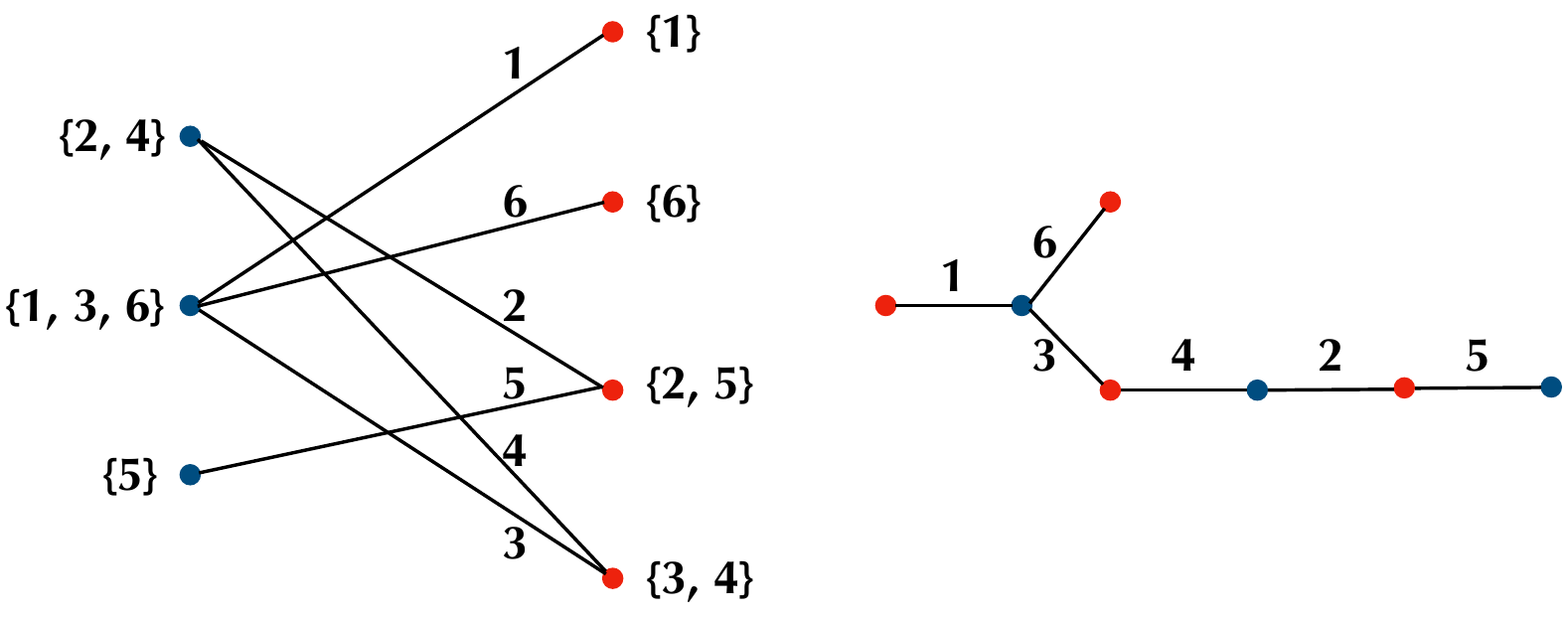}
\caption{Here we show the edge labeled spanning tree of $K_{3, 4}$ associated to the pair of sorted partitions $\sigma= (\{2, 4\}, \{1, 3, 6\}, \{5\})$ and $\tau = (\{1\}, \{6\}, \{2, 5\}, \{3, 4\})$.    }
 \label{fig:tree}
\end{figure}

Note that because $\ip{s}+\ip{t}=n+1$, by construction, $T_{(\sigma, \tau)}$ has $n+1$ vertices, $n$ edges, and the degree sequences of its bipartite components have types $s$ and $t$ respectively. Furthermore, the condition $\sigma \vee \tau=1_n$ implies that $T_{(\sigma, \tau)}$ is connected, hence $T_{(\sigma, \tau)}$ is a (spanning) tree (of $K_{s, t}$) whose edges are labelled with the numbers from 1 to $n$, that is $T_{(\sigma, \tau)}\in \mathcal{T}_{s, t}$. 

To reverse this procedure, start with $T \in \mathcal{T}_{s, t}$, and note that we can construct pair $(\sigma_T, \tau_T)$ of ordered partitions by defining the $i$-th block of $\sigma$ to be the set of the numbers assigned to the edges coming out of the $i$-th vertex in $s$ component of $K_{s, t}$ and similarly for $\tau$.  By construction it is clear that $(\sigma_T, \tau_T)\in \calP_{s, t}^{\text{or}}$ and that if we apply the procedure defined above to this pair we would obtain $T$. 
\bigskip

\emph{Step 2: Counting $|\mathcal{T}_{s, t}|$.} First we observe that we can count the number of spanning trees of $K_{s, t}$ with any prescribed degree subsequence. Indeed, if  $$\vec{d} = (d_1, \dots, d_{\ip{s}}; d_{\ip{s}+1}, \dots, d_{\ip{s}+\ip{t}} )$$
is a fixed degree sequence, one can use Pr\"ufer's trick to show that the set of spanning trees with degree sequence $\vec{d}$ is in 1-1 correspondence with the set of pairs of sequences $(\ell_s, \ell_t )$ satisfying that $\ell_s$ uses  numbers from $\ip{s}+1$ to $\ip{s}+\ip{t}$ and $\ell_t$ uses  numbers from $1$ to $\ip{s}$, and  each number $i$ appears exactly $d_i-1$ times in its corresponding sequence (we refer the reader to \cite[Pages 341-342]{hartsfield1990spanning} for a detailed description on how this is done). On the other hand it is clear that the number of pairs of sequences $(\ell_s, \ell_t)$ with these properties is 
$$\binom{\ip{s}-1} {d_1-1, \dots, d_{\ip{s}}-1 } \binom{\ip{t}-1}{ d_{\ip{s}+1}-1, \dots, d_{\ip{s}+\ip{t}-1}-1}. $$
Now note that the number of degree sequences of type $(s, t)$ is $\frac{\ip{s}!\ip{t}!}{p_sp_t}$ where $p_s:=s_1!\cdots s_n!$ and $p_t:=t_1!\cdots t_n!$. Hence, the number of spanning trees of $K_{s, t}$ with degree sequences of type $(s, t)$ is the product of the two aformentioned quantities, i.e.
$$S_{s, t}:=\frac{\ip{s}! \ip{t}! (\ip{s}-1)!(\ip{t}-1)!)}{p_s p_t (2!^{s_3} 3!^{s_4} \cdots(n- 1)!^{s_n})(2!^{t_3} 3!^{t_4} \cdots(n- 1)!^{t_n}) }. $$
Since $\mathcal{T}_{s, t}$ is the set of trees with the above properties, but also with labelled edges, we have $|\mathcal{T}_{s, t}| = n! S_{s, t} $. 
\bigskip

Finally, combining the discussions in steps 1 and 2 we get 
$$B(s, t) = \frac{|\calP_{s, t}^{\text{or}}|}{\ip{s}!\ip{t}!} = \frac{|\mathcal{T}_{s, t}|}{\ip{s}!\ip{t}!} = \frac{n! S_{s, t}}{\ip{s}!\ip{t}!} $$
and the proof is concluded. 
\end{proof}

\end{appendix}

\end{document}